\documentclass[reqno,tbtags,a4paper]{amsart}

\usepackage{amsmath,amssymb,amsthm,amsfonts,amscd,array,graphicx}
\usepackage{mathrsfs}

\newcommand{\x}{\mathbf{x}}
\newcommand{\V}{\mathcal{V}}
\newcommand{\F}{\mathcal{F}}
\newcommand{\E}{\mathcal{E}}
\newcommand{\R}{\mathbb{R}}
\newcommand{\C}{\mathbb{C}}
\newcommand{\Q}{\mathbb{Q}}
\newcommand{\Z}{\mathbb{Z}}
\newcommand{\supp}{\mathop{\mathrm{supp}}\nolimits}
\newcommand{\ch}{\mathop{\mathrm{char}}\nolimits}
\newcommand{\pf}{\mathop{\mathrm{pf}}\nolimits}

\newtheorem{theorem}{Theorem}[section]
\newtheorem{propos}[theorem]{Proposition}
\newtheorem{cor}[theorem]{Corollary}
\newtheorem{lem}[theorem]{Lemma}
\newtheorem*{mlem}{Main Lemma}
\theoremstyle{definition}
\newtheorem{defin}[theorem]{Definition}
\newtheorem{remark}[theorem]{Remark}
\newtheorem{quest}[theorem]{Question}

\author{Alexander A. Gaifullin}

\thanks{The work was partially supported by the Russian Foundation for Basic Research (projects 12-01-31444 and 12-01-92104), by a grant of the President of the Russian Federation (project MD-4458.2012.1), by a grant of the Government of the Russian Federation (project 2010-220-01-077),  by a programme of the Branch of Mathematical Sciences of the Russian Academy of Sciences, and by a grant from Dmitry Zimin's ``Dynasty'' foundation.}

\title[Generalization of Sabitov's Theorem to arbitrary dimensions]
{Generalization of Sabitov's Theorem to polyhedra of arbitrary dimensions}

\date{}

\address{Steklov Mathematical Institute, Moscow, Russia\newline
${}$\hspace{4.3mm}Moscow State University, Moscow, Russia\newline ${}$\hspace{4.3mm}Kharkevich Institute for Information Transmission Problems, Moscow, Russia\newline 
${}$\hspace{4.3mm}Yaroslavl State University, Yaroslavl, Russia}

\email{agaif@mi.ras.ru}

\begin{document}

\begin{abstract}
In 1996 Sabitov proved that the volume of an arbitrary simplicial polyhedron~$P$ in the $3$-dimensional Euclidean space~$\R^3$ satisfies a monic (with respect to~$V$) polynomial relation $F(V,\ell)=0$, where $\ell$ denotes the set of the squares of edge lengths of~$P$. In 2011 the author proved the same assertion for polyhedra in~$\R^4$. In this paper, we prove that the same result is true in arbitrary dimension $n\ge 3$. Moreover, we show that this is true not only for simplicial polyhedra, but for all polyhedra with triangular $2$-faces. As a corollary, we obtain the proof in arbitrary dimension of the well-known Bellows Conjecture posed by Connelly in~1978. This conjecture claims that the volume of any flexible polyhedron is constant. Moreover, we obtain the following stronger result. If $P_t$, $t\in [0,1]$, is a continuous deformation of a polyhedron such that the combinatorial type of~$P_t$ does not change and every $2$-face of~$P_t$ remains congruent to the corresponding face of~$P_0$, then the volume of~$P_t$ is constant. We also obtain non-trivial estimates for the oriented volumes of complex simplicial polyhedra in~$\C^n$ from their orthogonal edge lengths.
\end{abstract}

\maketitle
\section{Introduction}

It is well known that the area of a triangle can be computed from its side lengths by Heron's formula
$$
A^2=\frac{1}{16}\left(2a^2b^2+2b^2c^2+2c^2a^2-a^4-b^4-c^4\right).
$$
However, it is impossible to compute the area of a polygon with at least $4$ sides from its side lengths, since any such polygon can be flexed so that its area changes, while its  side lengths remain constant. 
The situation changes in dimension $3$.

\begin{theorem}[Sabitov~\cite{Sab96}, \cite{Sab98}]\label{theorem_Sabitov}
For any combinatorial type of simplicial polyhedra in the $3$-dimensional Euclidean space, there exists a polynomial relation
\begin{equation}\label{eq_Sabitov}
V^{2N}+a_1(\ell)V^{2N-2}+\dots+a_N(\ell)=0
\end{equation}
between the volume~$V$ of a polyhedron of the given combinatorial type and the set~$\ell$ of the squares of its edge lengths. This relation is monic with respect to $V$, and the coefficients $a_i(\ell)$ are polynomials with rational coefficients in the squares of edge lengths of the polyhedron. 
\end{theorem}

This theorem shows that the volume of a simplicial polyhedron is determined by its combinatorial type and edge lengths up to finitely many possibilities. The polynomial in the left-hand side of~\eqref{eq_Sabitov} is called a \textit{Sabitov polynomial\/} for  polyhedra of the given combinatorial type. It is convenient to assume that the relation~\eqref{eq_Sabitov} contains only even powers of~$V$, since we shall mean under~$V$ the oriented volume, which changes its sign whenever we reverse the orientation. Nevertheless, this is not a restriction. Indeed, if we had a monic polynomial relation $F(V,\ell)=0$ containing both even and odd powers of~$V$, then $F(V,\ell)F(-V,\ell)=0$ would be the required relation that contains only even powers of~$V$.

Theorem~\ref{theorem_Sabitov} was improved by Connelly, Sabitov, and Walz~\cite{CSW97}. They proved that for $W=12V$, there is a monic  relation
$$
W^{2N}+b_1(\ell)W^{2N-2}+\dots+b_N(\ell)=0
$$ 
such that $b_i$ are polynomials  with \textit{integral\/} coefficients.  In 2011 the author~\cite{Gai11} proved that the assertion of Theorem~\ref{theorem_Sabitov} holds as well for simplicial polyhedra in the $4$-dimensional Euclidean space. 
The main result of the present paper is the generalization of Theorem~\ref{theorem_Sabitov} to polyhedra in the Euclidean space of an arbitrary dimension. We shall prove the strong version of this theorem, i.\,e., the analogue of the result by Connelly, Sabitov, and Walz mentioned above.

Notice that the polyhedron is not supposed to be convex, and is not supposed to have  a combinatorial type of a sphere. Indeed, the polyhedron is not even supposed to be non-self-intersected and non-degenerate. To formulate our result rigorously, we need to give a definition of a possibly self-intersected and degenerate polyhedron.

First, we recall some basic definitions concerning simplicial complexes.
An \textit{abstract simplicial complex\/} on a vertex set~$\V$ is a set $K$ of subsets of~$\V$ such that $\emptyset\in K$, and $\tau\in K$ implies $\sigma\in K$ whenever $\sigma\subset\tau\subset \V$. (We use $\subset$ for non-strict inclusion.) By definition, $\dim\sigma$ is the cardinality of~$\sigma$ minus~$1$. To simplify the terminology, $n$-dimensional simplices will be called $n$-simplices. The geometric realization of a simplicial complex $K$ will be denoted by~$|K|$. 

A simplicial complex $K$ is called a $k$-dimensional \textit{pseudo-manifold\/} if every simplex of~$K$ is contained in a $k$-simplex, and every $(k-1)$-simplex of~$K$ is contained in  exactly two $k$-simplices.
$K$ is called \textit{strongly connected\/} if every two of its $k$-simplices can be connected by a sequence of $k$-simplices such that every two consecutive $k$-simplices have a common $(k-1)$-face. A pseudo-manifold is said to be \textit{oriented\/} if its $k$-simplices are endowed with consistent orientations. When we consider oriented simplices, the notation $\sigma=\{v_0,\ldots,v_k\}$  will mean that the simplex~$\sigma$ is endowed with the orientation given by the ordering $v_0,\ldots,v_k$ of its vertices.

\begin{defin}\label{def_ph1}
Let $K$ be an oriented $(n-1)$-dimensional strongly connected pseudo-manifold. A \textit{simplicial polyhedron\/} (or a \textit{polyhedral surface}) of combinatorial type~$K$ is a mapping $P\colon |K|\to\R^n$ such that the restriction of~$P$ to every simplex of~$|K|$ is linear. A polyhedron~$P$ is called \textit{embedded\/} if $P$ is injective.
\end{defin}

Notice that we do not require the restriction of~$P$ to a simplex of~$|K|$ to be non-degenerate. For instance, $P$ may map the whole simplex to a point.

Intuitively, we mean under a simplicial polyhedron in~$\R^n$ the  region bounded by  an embedded polyhedral surface $P(|K|)$. The oriented volume of this region will be denoted by $V(P)$ and will be called the oriented volume of~$P$.

However, non-embedded polyhedra are also useful in many problems especially concerning rigidity and flexibility. If $P\colon |K|\to\R^n$ is a non-embedded polyhedral surface, we may also try to consider some ``region bounded by it''. Though this region is not well defined as a subset of~$\R^n$, its oriented volume is well defined. The definition is as follows. For a point~$x\in\R^n\setminus P(|K|)$, consider a curve connecting a point $x$ with infinity, and denote by $\lambda(x)$ the algebraic intersection index of this curve and the singular surface $P(|K|)$.
Then $\lambda$ is an almost everywhere defined  piecewise constant function on~$\R^n$ with compact support. By definition, a \textit{generalized oriented volume\/}~$V(P)$ of a polyhedron~$P\colon |K|\to \R^n$ is equal to 
$
\int_{\R^n}\lambda(x)\,dV
$, where $dV=dx_1\dots dx_n$ is the standard volume element in~$\R^n$.

If $O\in\R^n$ is an arbitrary point, the generalized oriented volume~$V(P)$ can be equivalently defined by
$$
V(P)=\sum_{i=1}^qV\left(O,P\bigl(v_1^{(i)}\bigr),\ldots,P\bigl(v_n^{(i)}\bigr)\right),
$$
where $\{v_1^{(i)},\ldots,v_n^{(i)}\}$, $i=1,\ldots,q$, are all pairwise distinct positively oriented $(n-1)$-simplices of~$K$, and $V(p_0,\ldots,p_n)$ stands for the oriented volume of the simplex with vertices~$p_0,\ldots,p_n$. It can be easily checked that $V(P)$ is independent of the choice of~$O$.

The main result of this paper is as follows.
  
\begin{theorem}\label{theorem_main}
Suppose, $n\ge 3$. Let $K$ be an $(n-1)$-dimensional oriented strongly connected pseudo-manifold. Then there exists a monic polynomial relation
$$
W^{2N}+b_1(\ell)W^{2N-2}+\dots+b_N(\ell)=0
$$
that holds for all polyhedra $P\colon |K|\to\R^n$ of combinatorial type~$K$. Here  $W=2^{[\frac{n}{2}]}\,n!\,V(P)$,  and
 $b_i(\ell)$ are polynomials with integral coefficients in the squares of edge lengths of~$P$. 
\end{theorem}

The most important application of Sabitov's Theorem concerns flexible polyhedra. A \textit{flex} of a simplicial polyhedron~$P\colon |K|\to\R^n$ is a continuous family of polyhedra $P_t\colon |K|\to\R^n$ of the same combinatorial type such that $P_0=P$, the edge lengths of~$P_t$ are constant, and $P_0$ cannot be taken to~$P_1$ by an isometry of~$\R^n$. A polyhedron is called \textit{flexible\/} if it admits a flex, and is called \textit{rigid\/} if it does not admit a flex. 
In 1896 Bricard~\cite{Bri97} constructed his famous self-intersected flexible octahedra in~$\R^3$. The  first example of an embedded flexible polyhedron in~$\R^3$ was constructed by Connelly~\cite{Con77}; the simplest known example belongs to Steffen. 
Examples of flexible cross-polytopes in~$\R^4$ were constructed by Walz (cf.~\cite{Pak08}) and Stachel~\cite{Sta00}. All these examples are not embedded. The problem on existence of embedded flexible polyhedra in~$\R^4$ and the problem on existence of non-trivial flexible polyhedra in~$\R^5$ remain open.

The Bellows Conjecture due to Connelly~\cite{Con78} claims that the generalized oriented volume of any flexible polyhedron remains constant under the flex.  
It follows from Theorem~\ref{theorem_Sabitov} that the volume of a simplicial polyhedron in~$\R^3$ with the given combinatorial type and edge lengths can take only finitely many values, which implies that the volume cannot change continuously. Notice that it is very important here that the polynomial relation~\eqref{eq_Sabitov} is monic with respect to $V$. Otherwise, it could happen that for some particular values of the edge lengths, all coefficients of~\eqref{eq_Sabitov} would vanish, hence, this relation would impose no restriction on~$V$.

\begin{cor}[Sabitov~\cite{Sab96}]\label{cor_Sabitov}
The Bellows Conjecture holds for flexible polyhedra in~$\R^3$.
\end{cor}

In the same way, Theorem~\ref{theorem_main} implies

\begin{cor}
The Bellows Conjecture holds for flexible polyhedra in~$\R^n$, $n\ge 3$. 
\end{cor}

Unfortunately, the importance of this result in higher dimensions is reduced by the lack of examples of flexible polyhedra.

Sometimes it is convenient to work with a more general definition of a simplicial polyhedron in~$\R^n$, than Definition~\ref{def_ph1}. Consider an (abstract) simplex~$\Delta^M$ of large dimension~$M$. Let $(C_*(\Delta^M),\partial)$ be the simplicial chain complex of~$\Delta^M$ with integral coefficients. This means that $C_k(\Delta^M)$ is the free Abelian group generated by oriented $k$-faces of~$\Delta^M$, and $\partial\colon C_k(\Delta^M)\to C_{k-1}(\Delta^M)$ is the boundary operator. A chain $Z\in C_k(\Delta^M)$ is called a \textit{cycle} if $\partial Z=0$. The \textit{support} of a chain~$Z$ is the subcomplex $\supp(Z)\subset\Delta^M$ consisting of all $k$-simplices that enter~$Z$ with non-zero coefficients, and all their faces.

\begin{defin}\label{def_ph2}
 A \textit{simplicial polyhedron\/}  is a pair $(Z,P)$, where $Z$ is an $(n-1)$-cycle and $P\colon |\supp(Z)|\to\R^n$ is a mapping whose restriction to every simplex of~$|\supp(Z)|$ is linear. We say that the polyhedron~$(Z,P)$ has \textit{combinatorial type\/}~$Z$.
  \textit{Edges\/} of the polyhedron~$(Z,P)$ are the images of edges of~$|\supp(Z)|$ under~$P$.
\end{defin}

The \textit{generalized oriented volume\/}~$V_Z(P)$ of a polyhedron~$(Z,P)$ is defined by 
$
\int_{\R^n}\lambda(x)\,dV
$, where $\lambda(x)$ is the algebraic intersection index of a curve from $x$ to infinity and the singular cycle $P(Z)$. Equivalently, for an arbitrary point $O\in\R^n$,
$$
V_Z(P)=\sum_{i=1}^qc_iV\left(O,P\bigl(v_1^{(i)}\bigr),\ldots,P\bigl(v_n^{(i)}\bigr)\right),
$$
where $Z=\sum_{i=1}^qc_i\{v_1^{(i)},\ldots,v_n^{(i)}\}$, $c_i\in\Z$.

\begin{theorem}\label{theorem_main2}
Suppose, $n\ge 3$. Let $Z$ be an $(n-1)$-cycle. Then there exists a monic polynomial relation
$$
W^{2N}+b_1(\ell)W^{2N-2}+\dots+b_N(\ell)=0
$$
that holds for all polyhedra $P\colon |\supp(Z)|\to\R^n$ of combinatorial type~$Z$. Here  $W=2^{[\frac{n}{2}]}\,n!\,V_Z(P)$,  and
 $b_i(\ell)$ are polynomials with integral coefficients in the squares of edge lengths of~$P$. 
\end{theorem}

Obviously, Definition~\ref{def_ph1} is contained in Definition~\ref{def_ph2}. Indeed, every polyhedron of combinatorial type~$K$ for an oriented $(n-1)$-dimensional  pseudo-manifold~$K$ is a polyhedron of combinatorial type~$[K]$, where $[K]$ is the fundamental cycle of~$K$. Hence Theorem~\ref{theorem_main} follows from Theorem~\ref{theorem_main2}.

In fact, an analogue of Theorem~\ref{theorem_main} holds for a wider class of polyhedra than simplicial polyhedra, namely, for all polyhedra with triangular $2$-faces.
To formulate this result rigorously, we need to give a rigorous definition of a not necessarily simplicial polyhedron. We prefer to use the following standard geometric definition, which does not include any degenerate or self-intersected examples (cf.~\cite{Ede95}). (Actually, a more general definition in spirit of Definitions~\ref{def_ph1} and~\ref{def_ph2} could also be given, but it would make our considerations less clear.) 

An $n$-\textit{polyhedron\/} is a subset $P\subset\R^n$ with connected interior that can be obtained as the union of finitely many $n$-dimensional convex polytopes. (We only require the existence of such convex polytopes. No particular convex polytopes are supposed to be chosen.)

The \textit{point figure\/} $\pf(p)$ of a point $p\in P$ is the set consisting of all points $q\in\R^n$ such that the point $(1-\varepsilon)p+\varepsilon q$ belongs to~$P$ for all sufficiently small positive $\varepsilon$.
For each point $p\in P$, consider the set of all points $q\in P$ such that $\pf(q)=\pf(p)$ and choose the connected component of this set containing~$p$. The closure of this connected component is called a \textit{face\/} of~$P$.  In particular, $P$ is a face of itself. All other faces of~$P$ are called \textit{proper\/}. It is easy to check that 
\begin{itemize}
\item $P$ has finitely many faces each of which is a $k$-polyhedron in a certain $k$-plane in~$\R^n$.
\item The boundary of every face is the union of faces of smaller dimensions.
\item If $F_1$ and $F_2$ are distinct faces of the same dimension, then the intersection  $F_1\cap F_2$ either is empty or is the union of faces of smaller dimensions.
\end{itemize}
Notice, that it is possible that the relative interiors of faces are not disjoint. Zero-dimensional and one-dimensional faces are called \textit{vertices\/} and \textit{edges\/} respectively.
  
Let $\F$ be the set of all faces of~$P$. We introduce a partial ordering on~$\F$ in the following way. We put $F_1<F_2$ if and only if there is a sequence of faces $G_1=F_1$, $G_2$, $\ldots$, $G_k=F_2$, $k\ge 2$, such that $\dim G_{i+1}=\dim G_{i}+1$ and $G_i\subset\partial G_{i+1}$. Notice that $F_1\not< F_2$ if the relative interior of~$F_1$ is contained in the relative interior of~$F_2$. The poset~$\F$ is called the \textit{face poset\/} of~$P$. We shall say that $P$  has  \textit{combinatorial type\/}~$\F$.

\begin{theorem}\label{theorem_main_ns}
Suppose, $n\ge 3$. Let $\F$ be a finite poset that can be realized as the face poset of an $n$-polyhedron with triangular $2$-faces. Then there exists a monic polynomial relation
\begin{equation}
\label{eq_rel_ns}
W^{2N}+b_1(\ell)W^{2N-2}+\dots+b_N(\ell)=0
\end{equation}
that holds for all $n$-polyhedra $P$ of combinatorial type~$\F$. Here  $W$ is the volume of~$P$ multiplied by~$2^{[\frac{n}{2}]}n!$,  and $b_i(\ell)$ are polynomials with integral coefficients in the squares of edge lengths of~$P$. 
\end{theorem}

For non-simplicial polyhedra, the Bellows Conjecture says that for each deformation $P_t$, $t\in [0,1]$, such that the combinatorial type of~$P_t$ does not change and all proper faces of~$P_t$ remain congruent to the corresponding faces of~$P_0$, the volume of~$P_t$ is constant. Theorem~\ref{theorem_main_ns} allows us to prove the following result, which is stronger then the Bellows Conjecture.

\begin{cor}\label{cor_flex_ns}
Let $P_t$, $t\in [0,1]$, be a continuous deformation of an $n$-polyhedron such that the combinatorial type of~$P_t$ does not change and all $2$-dimensional faces of~$P_t$ remain congruent to the corresponding faces of~$P_0$. Then the volume of~$P_t$ is constant.
\end{cor}

Let us discuss main ideas of our proofs of Theorems~\ref{theorem_main}, \ref{theorem_main2}, and~\ref{theorem_main_ns}. Until now there existed two different proofs of Theorem~\ref{theorem_Sabitov} (i.\,e. Theorem~\ref{theorem_main} in dimension~$3$), Sabitov's original proof (see~\cite{Sab96}, \cite{Sab98}) and the proof due to Connelly, Sabitov, and Walz~\cite{CSW97}. Both proofs use the same induction on certain parameters of the polyhedron (genus, number of vertices, etc.). In the proof due to Sabitov, the induction step is based on a special procedure for elimination of diagonal lengths of the polyhedron by means of resultants. In the proof due to Connelly, Sabitov, and Walz this complicated elimination procedure is replaced with the usage of theory of places. A \textit{place\/} is a mapping~$\varphi\colon E\to F\cup\{\infty\}$ that satisfies certain special conditions, where $E$ and $F$ are fields  (cf. section~\ref{section_places}). Theory of places is applied as follows. For $E$ we take the field of rational functions in the coordinates of vertices of the polyhedron. Now a relation of the form~\eqref{eq_Sabitov} exists if and only if every place of~$E$ that is finite on the squares of edge lengths of the polyhedron is also finite on its volume. In~\cite{Gai11}, the author generalized both proofs to dimension~$4$. The induction became more delicate, and an additional result of algebraic geometry was used.

In the present paper our approach is different. Instead of proceeding by induction on some parameters of the polyhedron, we study more carefully the properties of places of the field~$E$. Namely, for every place~$\varphi$, we study the  graph consisting of all edges and diagonals of the polyhedron whose squares of lengths have finite $\varphi$-values. We prove that the clique simplicial complex of this graph collapses on a subcomplex of dimension not greater than~$\frac{n}{2}$.  This implies that the $(n-1)$-dimensional homology group of this complex vanishes, which allows us to prove Theorems \ref{theorem_main2} and~\ref{theorem_main_ns}. Notice that our proof is quite elementary modulo theory of places. In particular, unlike the proof in~\cite{Gai11}, it does not use any results of algebraic geometry.

This paper is organized as follows. In section~\ref{section_alg} we give a convenient algebraic reformulation of Theorem~\ref{theorem_main2} (Theorem~\ref{theorem_alg}). The same reformulation was used by Connelly, Sabitov, and Walz in
dimension~$3$, and by the author in dimension~$4$ (see \cite{CSW97},~\cite{Gai11}). In sections~\ref{section_places} and~\ref{section_CM} we recall two main tools of our proof, theory of places and the Cayley--Menger determinants respectively. In section~\ref{section_scheme} we  give the scheme of proof of Theorem~\ref{theorem_alg}. We formulate a lemma, which we call Main Lemma, and derive Theorem~\ref{theorem_alg} from this lemma. The proof of Main Lemma is contained in sections~\ref{section_ord} and~\ref{section_proof_lemma}. In section~\ref{section_complex} we apply Theorem~\ref{theorem_main2} to obtain estimates for the volumes of polyhedra in~$\C^n$. Section~\ref{section_triang} is devoted to non-simplicial polyhedra with triangular $2$-faces, and contains the proofs of Theorem~\ref{theorem_main_ns} and Corollary~\ref{cor_flex_ns}.

The author is grateful to I.\,I.\,Bogdanov, S.\,O.\,Gorchinsky, I.\,Pak, and I.\,Kh.\,Sabitov for useful discussions.

\section{Algebraic reformulation}\label{section_alg}

In this paper, a ring is always a commutative ring with unity.

Let $Z$ be an  $(n-1)$-cycle, and let $K$ be its support. Let $
\mathcal V$ be the vertex set of~$K$, and let $m$ be the cardinality of~$\mathcal V$. A polyhedron $P\colon |K|\to\R^n$ is uniquely determined by the images $P(v)$ of the vertices $v\in \mathcal{V}$. We denote by $x_{v,1},\ldots,x_{v,n}$ the coordinates of the point~$P(v)$. Now it is convenient to regard $x_{v,i}$, $v\in \V$, $i=1,\ldots,n$, as independent variables, and to interpret them as coordinates on the space of all polyhedra~$P$ of combinatorial type~$K$. We denote by $\x_{\V}$ the set of $mn$ variables~$x_{v,i}$ and by $\Z[\x_{\V}]$ the ring of polynomials in these variables. For any two vertices $u,v\in\V$, we define the \textit{universal square of the distance}  between $u$ and~$v$ to be the element~$\ell_{uv}\in\Z[\x_{\V}]$ given by
$$
\ell_{uv}=\sum_{i=1}^n(x_{u,i}-x_{v,i})^2.
$$
We shall say that $\ell_{uv}$ is the \textit{universal square of the edge length} if $\{u,v\}\in K$, and the \textit{universal square of the diagonal length}  if $\{u,v\}\notin K$.
Let $R_K$ be the subring of~$\Z[\x_{\V}]$ generated by all $\ell_{uv}$ such that $\{u,v\}\in K$.

The \textit{universal oriented volume\/} $V_Z\in\Q[\x_{\V}]$ is defined by 
$$
V_Z=\frac{1}{n!}\sum_{i=1}^qc_i\begin{vmatrix}
x_{v_1^{(i)},1}&\cdots&x_{v_n^{(i)},1}\\
\vdots&\ddots&\vdots\\
x_{v_1^{(i)},n}&\cdots&x_{v_n^{(i)},n}
\end{vmatrix},
$$ 
where $Z=\sum_{i=1}^qc_i\big\{v_1^{(i)},\ldots,v_n^{(i)}\big\}$, $c_i\in\Z$. 

We put $W_Z=2^{[\frac{n}{2}]}\,n!\,V_Z$. Then $W_Z\in\Z[\x_{\V}]$.

Polyhedra $P\colon |K|\to\R^n$ are in one-to-one correspondence with  specialization homomorphisms $\Z[\x_{\V}]\to\R$. Any such specialization homomorphism takes the universal squares of edge and diagonal lengths~$\ell_{uv}$, and the universal oriented volume~$V_Z$ to the squares of edge and diagonal lengths, and the oriented volume of the corresponding polyhedron~$(Z,P)$.

A polynomial relation among the polynomials $\ell_{uv}$, $\{u,v\}\in K$, and $W_K$ holds in~$\Z[\x_{\V}]$ if and only if it holds after substituting any real values for the variables~$x_{v,i}$. 
Hence Theorem~\ref{theorem_main2} is equivalent to the following one.

\begin{theorem}\label{theorem_alg}
Let $Z$ be an $(n-1)$-cycle and let $K$ be its support, $n\ge 3$. Then the element~$W_Z$ is integral over the ring~$R_K$.
\end{theorem}

Denote by $\Q(\x_{\V})$ the field of rational functions in the variables~$x_{v,i}$.
It will be useful for us to consider points $z\in\Q(\x_{\V})^n$ not necessarily coinciding with the vertices $v\in \V$. We shall conveniently define the coordinates of~$z$ by $x_{z,1},\ldots,x_{z,2}$. Here $x_{z,i}$ are elements of~$\Q(\x_{\V})$, which are not supposed to be independent variables. The square of the distance $\ell_{yz}\in\Q(\x_{\V})$ between two points $y,z\in\Q(\x_{\V})^n$ is given by the same formula
\begin{equation}\label{eq_dist}
\ell_{yz}=\sum_{i=1}^n(x_{y,i}-x_{z,i})^2.
\end{equation}

\section{Places}\label{section_places}

The main tool of our proof is theory of places. In dimension~$3$, it was used 
by Connelly, Sabitov, and Walz  to simplify Sabitov's  proof of Theorem~\ref{theorem_Sabitov}.
Let us recall necessary facts on places and valuations (cf. \cite[chapter 1]{Lan72}, \cite[section 41.7]{Pak08}).

Let $E$ be a field. A \textit{place} of~$E$ is a mapping $\varphi\colon E\to F\cup\{\infty\}$ to a field~$F$, with an extra element~$\infty$, such that $\varphi(1)=1$, $\varphi(a+b)=\varphi(a)+\varphi(b)$ and $\varphi(ab)=\varphi(a)\varphi(b)$ whenever the right-hand sides are defined. Here we assume that $1/0=\infty$, $1/\infty=0$, $c\pm\infty=\infty$ for all $c\in F$, and $c\cdot\infty=\infty$ for all $c\in F\cup\{\infty\}$, except for~$0$. The expressions $\infty\pm\infty$, $\infty\cdot 0$, $0/0$, and $\infty/\infty$ are undefined. Elements $c\in F$ are said to be \textit{finite\/}.

By definition, the \textit{valuation ring\/} of~$\varphi$ is the subring $\mathfrak{o}\subset E$ consisting of all $a\in E$ such that $\varphi(a)\ne\infty$. Let $E^*$ be the multiplicative group of the field~$E$, and let $U$ be the group of units of~$\mathfrak{o}$. We put $\Gamma=E^*/U$. Let $\gamma_1,\gamma_2$ be elements of~$\Gamma$ corresponding to cosets $a_1U$ and~$a_2U$ respectively. It is easy to see that if $\gamma_1\ne\gamma_2$, then $\varphi(a_1/a_2)$ is either $0$ or $\infty$. We put $\gamma_1<\gamma_2$ if $\varphi(a_1/a_2)=0$, and $\gamma_1>\gamma_2$ if $\varphi(a_1/a_2)=\infty$. It is easy to see that $<$ is a well-defined ordering on the group~$\Gamma$, in particular, $\gamma_1<\gamma_2$ implies $\alpha\gamma_1<\alpha\gamma_2$ for any $\alpha\in\Gamma$. We add to the ordered group~$\Gamma$ an extra element~$0$, and put $0\cdot\gamma=0$ and $0<\gamma$ for all $\gamma\in\Gamma$.

For any element $a\in E^*$, let $|a|_{\varphi}$ be the coset $aU\in\Gamma$. We put $|0|_{\varphi}=0$. The mapping $E\to\Gamma\cup\{0\}$ given by $a\mapsto|a|_{\varphi}$ is called the \textit{valuation\/} corresponding to the place~$\varphi$, and satisfies the following properties:
\begin{itemize}
\item $|a|_{\varphi}=0$ if and only if $a=0$.
\item $|ab|_{\varphi}=|a|_{\varphi}\cdot|b|_{\varphi}$.
\item $|a+b|_{\varphi}\le \max(|a|_{\varphi},|b|_{\varphi})$.
\item $|a|_{\varphi}<1$ iff $\varphi(a)=0$, $|a|_{\varphi}=1$ iff $\varphi(a)\in F\setminus\{0\}$, and $|a|_{\varphi}>1$ iff $\varphi(a)=\infty$.
\end{itemize}
  
The most important for us fact on places is as follows.
\begin{lem}[cf.~{\cite[p. 12]{Lan72}}]\label{lem_place}
Let $R$ be a ring contained in a field~$E$, and let $a$ be an element of~$E$. Then $a$ is integral over~$R$ if and only if every place~$\varphi$ of~$E$ that is finite on~$R$ is finite on~$a$. 
\end{lem}

\section{Cayley--Menger determinants}\label{section_CM}

\begin{propos}[Cayley~\cite{Cay41}, Menger~\cite{Men28}, \cite{Men31}, cf.~\cite{Pak08}]
Let $v_0,v_1,\ldots,v_n$ be points in a Euclidean space and let $\ell_{v_iv_j}$ be the square of the distance between~$v_i$ and~$v_j$. Then the $n$-dimensional volume~$V$ of the simplex with vertices $v_0,v_1,\ldots,v_n$ satisfy the equation
\begin{equation}\label{eq_CMvol}
V^2=\frac{(-1)^{n+1}}{2^n(n!)^2}CM(v_0,\ldots,v_n),
\end{equation}
where
\begin{equation}\label{eq_CM}
CM(v_0,\ldots,v_n)=\left|
\begin{matrix}
0 & 1 & 1 & 1 & \cdots & 1\\
1 & 0 & \ell_{v_0v_1} & \ell_{v_0v_2} & \cdots & \ell_{v_0v_n}\\
1 & \ell_{v_0v_1} & 0 & \ell_{v_1v_2} & \cdots & \ell_{v_1v_n}\\
1 & \ell_{v_0v_2} & \ell_{v_1v_2} & 0 & \cdots & \ell_{v_2v_n}\\
\vdots & \vdots & \vdots & \vdots & \ddots & \vdots\\
1 & \ell_{v_0v_n} & \ell_{v_1v_n} & \ell_{v_2v_n} & \cdots & 0
\end{matrix}
\right|\ .
\end{equation} 
\end{propos}

The determinant $CM(v_0,\ldots,v_n)$ is called the \textit{Cayley--Menger determinant\/} of the points $v_0,\ldots,v_n$.

\begin{cor}\label{cor_CM}
The Cayley--Manger determinant of any affinely dependent set of points vanishes. In particular, the Cayley--Menger determinant of any $n+2$ points in~$\R^n$ vanishes.
\end{cor}

The relation $CM(v_0,\ldots,v_k)=0$ for affinely dependent points $v_0,\ldots,v_k$ is called the \textit{Cayley--Menger relation\/}.

If $v_0,\ldots,v_n\in\R^n$, formula~\eqref{eq_CMvol} is a polynomial relation among the coordinates of the points~$v_0,\ldots,v_n$, hence this formula holds as well for universal $n$-dimensional points $v_0,\ldots,v_n$, where we substitute for~$\ell_{v_iv_j}$ the universal squares of distances and for $V$ the universal oriented volume of the $n$-simplex with vertices $v_0,\ldots,v_n$. Notice, that in our notation introduced in section~\ref{section_alg} this universal oriented volume is denoted by~$V_{\partial\Delta}$, where $\Delta=\{v_0,\ldots,v_n\}$. Now consider $W_{\partial\Delta}=2^{[\frac{n}{2}]}\, n!\,V_{\partial\Delta}$. If $n$ is even, we have $W_{\partial\Delta}=-CM(v_0,\ldots,v_n)$. Hence $W_{\partial\Delta}$ is integral over the ring $R_{\Delta}$ generated by all universal squares of  edge lengths~$\ell_{v_iv_j}$. If $n$ is odd, we have $W_{\partial\Delta}=\frac12CM(v_0,\ldots,v_n)$. But the determinant of a  symmetric matrix of odd order with zero diagonal is a polynomial in the matrix entries with \textit{even\/} coefficients. Hence $\frac12CM(v_0,\ldots,v_n)$ is the polynomial in~$\ell_{v_iv_j}$ with integral coefficients. Therefore $W_{\partial\Delta}$ is integral over~$R_{\Delta}$. Thus the claim of Theorem~\ref{theorem_alg} holds if $Z=\partial\Delta$.

In fact, it follows easily from formula~\eqref{eq_CMvol} that $kV_{\partial\Delta}$ is not integral over~$R_{\Delta}$ if $k\in\Z$, $0<k<2^{[\frac{n}{2}]}\,n!$. Hence the multiple $2^{[\frac{n}{2}]}\,n!$ in Theorems~\ref{theorem_main}, \ref{theorem_main2}, and~\ref{theorem_main_ns} cannot be decreased.

Notice also that Corollary~\ref{cor_CM} holds for points over any field~$E$, where the squares of the distances~$\ell_{v_iv_j}$ are defined by~\eqref{eq_dist}.

\section{Scheme of proof of Theorem~2.1}\label{section_scheme}

We shall apply Lemma~\ref{lem_place} to prove Theorem~\ref{theorem_alg}. So we take for~$E$, $R$, and~$a$ the field~$\Q(\x_{\V})$, the ring~$R_K$, and the element~$W_Z$ respectively. Theorem~\ref{theorem_alg} will follow, if we prove that every place~$\varphi$ of~$\Q(\x_{\V})$ that is finite on~$R_K$ is finite on~$W_Z$.

Let $\varphi\colon\Q(\x_{\V})\to F\cup\{\infty\}$ be a place, and let $|\cdot|_{\varphi}$ be the corresponding valuation on~$\Q(\x_{\V})$. Let $G_{\varphi}$ be the graph on the vertex set~$\V$ such that an edge $\{u,v\}$ belongs to $G_{\varphi}$ if and only if $\varphi(\ell_{uv})$ is finite. Let $K_{\varphi}$ be the \textit{clique simplicial complex\/} of~$G_{\varphi}$ (or the \textit{full subcomplex\/} spanned by~$G_{\varphi}$). This means that a simplex $\sigma\subset \V$ belongs to $K_{\varphi}$ if and only if all edges of~$\sigma$ belong to~$G_{\varphi}$.

Let $L$ be a simplicial complex. A pair $(\sigma,\tau)$ of simplices of~$L$ is called a \textit{free pair\/} if 
\begin{itemize}
\item $\tau\subset\sigma$, and $\dim\tau=\dim\sigma-1$.
\item $\sigma$ is maximal in~$L$, i.\,e., $\sigma$ is strictly contained in no simplex of~$L$.
\item $\tau$ is a \textit{free face\/} of~$\sigma$, i.\,e., $\tau$ is strictly contained in no simplex of~$L$ distinct from~$\sigma$. 
\end{itemize}
Removing a free pair $(\sigma,\tau)$ from~$L$, we obtain another simplicial complex~$L_1$, which is called an \textit{elementary collapse\/} of~$L$. A sequence of elementary collapses is called a \textit{collapse\/}. If $L$ collapses on $J$, then $L$ is homotopy equivalent to~$J$.

\begin{mlem}
For any place $\varphi\colon \Q(\x_{\V})\to F\cup\{\infty\}$, the simplicial complex~$K_{\varphi}$ collapses on a subcomplex of dimension less than or equal to~$\bigl[\frac{n}{2}\bigr]$.  
\end{mlem}

\begin{cor}\label{cor_main} 
$H_k(K_{\varphi})=0$ for $k>\frac{n}{2}$. In particular, $H_{n-1}(K_{\varphi})=0$ if $n\ge 3$. 
\end{cor}

The proof of Main Lemma will be given in sections~\ref{section_ord},~\ref{section_proof_lemma}. Now we shall prove that Theorem~\ref{theorem_alg} follows from Main Lemma.

\begin{proof}[Proof of Theorem~\ref{theorem_alg}]
To prove that $W_Z$ is integral over~$R_K$, we need to show that every place $\varphi\colon \Q(\x_{\V})\to F\cup\{\infty\}$ that is finite on~$R_K$ is also finite on~$W_Z$. Since $\varphi$ is finite on all $\ell_{uv}$ such that $\{u,v\}$ is an edge of~$K$, we see that $K\subset K_{\varphi}$. By Corollary~\ref{cor_main}, we have $H_{n-1}(K_{\varphi})=0$. Hence there exists a chain $Y\in C_n(K_{\varphi})$ such that $\partial Y=Z$. Suppose, $Y=\sum_{i=1}^kc_i\Delta_i$, where $c_i\in\Z$ and $\Delta_i$ are oriented $n$-simplices of~$K_{\varphi}$. Then 
$$
W_Z=\sum_{i=1}^kc_iW_{\partial \Delta_i}.
$$
The Cayley-Menger formula implies that $W_{\partial \Delta_i}$ is integral over the ring $R_{\Delta_i}$ generated by all $\ell_{uv}$ such that $u,v\in\Delta_i$ (cf.~section~\ref{section_CM}). Since $\Delta_i\in K_{\varphi}$, we see that $\varphi$ is finite on~$R_{\Delta_i}$. Hence $\varphi$ is finite on every~$W_{\partial \Delta_i}$. Therefore $\varphi$ is finite on~$W_Z$. 
\end{proof}

\section{Orderings~$\succ$}\label{section_ord}

We denote the set of all $k$-simplices of~$K_{\varphi}$ by $K_{\varphi}^k$. Suppose that for every~$k$, $0\le k\le\dim K_{\varphi}$, we have chosen a (strict) total ordering~$\succ$ on the set~$K^k_{\varphi}$. (It is convenient to denote the orderings on different sets~$K^k_{\varphi}$ by the same symbol~$\succ$.) Later we shall construct the orderings~$\succ$ that  will satisfy certain special conditions.
To formulate these conditions we need to introduce some notation.

A \textit{facet\/} of a simplex is a codimension~$1$ face of it.
For each $\sigma\in K_{\varphi}$
 such that $\dim\sigma>0$, we denote by $\mu(\sigma)$ the maximal  facet of~$\sigma$ with respect to $\succ$.

Suppose $\sigma\in K_{\varphi}$, $\sigma\ne\emptyset$. Denote by $\V_{\sigma}$ the set  consisting of all vertices $v\notin\sigma$ such that $\sigma\cup\{v\}\in K_{\varphi}$ and $\mu(\sigma\cup\{v\})=\sigma$. For each vertex $v\in \V_{\sigma}$, we denote by $\V_{\sigma}(v)$ the subset of~$\V_{\sigma}$ consisting of all vertices~$u$ such that $\sigma\cup\{v\}\succ\sigma\cup\{u\}$.

The first condition imposed on~$\succ$ is as follows. 

\begin{itemize}
\item[(i)]  If $\sigma_1,\sigma_2\in K_{\varphi}$, $\dim\sigma_1=\dim\sigma_2>0$, and  $\mu(\sigma_1)\succ\mu(\sigma_2)$, then $\sigma_1\succ\sigma_2$.
\end{itemize}

Other conditions will be slightly different in the cases of odd and even~$n$. 
If $n$ is odd, the ordering~$\succ$ must satisfy the following condition:

\begin{itemize}
\item[(ii)${}'$] Suppose that $\sigma\in K_{\varphi}$, $\sigma\ne\emptyset$, $v\in \V_{\sigma}$, and $\V_{\sigma}(v)\ne \emptyset$.  Then there is a vertex $u\in \V_{\sigma}(v)$ such that
$$
\varphi\left(\frac{\ell_{w_1w_2}}{\ell_{uv}}\right)\ne\infty
$$
for all $w_1,w_2\in \V_{\sigma}(v)\cup\{v\}$.
\end{itemize}
If $n$ is even, condition~(ii)${}'$ is replaced with the following two conditions:
\begin{itemize}
\item[(ii)${}''$] Suppose that $\sigma\in K_{\varphi}$, $\dim\sigma>0$, $v\in \V_{\sigma}$, and $\V_{\sigma}(v)\ne \emptyset$.  Then there is a vertex $u\in \V_{\sigma}(v)$ such that
$$
\varphi\left(\frac{\ell_{w_1w_2}}{\ell_{uv}}\right)\ne\infty
$$
for all $w_1,w_2\in \V_{\sigma}(v)\cup\{v\}$.
\item[(iii)${}''$] Let $\{u,v\}$ and $\{u,w\}$ be edges of~$K_{\varphi}$ such that $\{u\}\succ\{v\}$, $\{u\}\succ\{w\}$, and $\{u,v\}\succ\{u,w\}$. Then $|x_{u,1}-x_{v,1}|_{\varphi}\ge |x_{u,1}-x_{w,1}|_{\varphi}$.
\end{itemize}

Notice that condition~(ii)${}''$ almost coincide with condition~(ii)${}'$. The only difference is that in~(ii)${}''$ we assume that $\dim\sigma>0$.

\textit{Construction of $\succ$.\/} Now we shall construct the orderings~$\succ$ on the sets~$K_{\varphi}^k$ satisfying conditions~(i), (ii)${}'$ for odd~$n$ and conditions~(i), (ii)${}''$, (iii)${}''$ for even~$n$.

First, we choose an arbitrary total ordering~$\succ$ on the set~$K_{\varphi}^0$ of vertices of~$K_{\varphi}$. 

Second, suppose that $n$ is even and construct the ordering~$\succ$ on the set~$K_{\varphi}^1$ of edges of~$K_{\varphi}$. Let $\{u_1,v_1\}$ and $\{u_2,v_2\}$ be any edges of~$K_{\varphi}$ such that $\{u_1\}\succ\{v_1\}$ and $\{u_2\}\succ\{v_2\}$. We put $\{u_1,v_1\}\succ\{u_2,v_2\}$ whenever $\{u_1\}\succ\{u_2\}$. Then condition~(i) is satisfied. Now for every $u$, we need to order all edges $\{u,v\}$ such that $\{u\}\succ\{v\}$. We put $\{u,v_1\}\succ\{u,v_2\}$ whenever $\{u\}\succ\{v_1\},\{v_2\}$, and $|x_{u,1}-x_{v_1,1}|_{\varphi}>|x_{u,1}-x_{v_2,1}|_{\varphi}$. The edges $\{u,v\}$ with the same maximal vertex~$u$ and the same value of~$|x_{u,1}-x_{v,1}|_{\varphi}$ are ordered arbitrarily.  Then condition~(iii)${}''$ is satisfied.

Further, we proceed by induction on dimension. 
Suppose that we have already determined the ordering~$\succ$ on the set~$K_{\varphi}^{k-1}$ and construct the ordering~$\succ$ on the set~$K_{\varphi}^{k}$. We assume that $k\ge 1$ if $n$ is odd, and $k\ge 2$ if $n$ is even, since the case $k=1$ for $n$ even has been considered separately above.  

For $k$-simplices $\sigma_1,\sigma_2\in K_{\varphi}$, we put $\sigma_1\succ\sigma_2$ whenever $\mu(\sigma_1)\succ\mu(\sigma_2)$. Then condition~(i) is satisfied.
Now, for each $(k-1)$-simplex $\rho\in K_{\varphi}$, we need to order $k$-simplices $\sigma\in K_{\varphi}$ such that $\mu(\sigma)=\rho$. Each such $k$-simplex $\sigma$ has the form $\rho\cup\{v\}$, where $v\in \V_{\rho}$. Let $r$ be the cardinality of~$\V_{\rho}$. 
We shall successively denote the vertices of~$\V_{\rho}$ by $v_1,\ldots,v_r$ in the following way. Suppose that the vertices $v_1,\ldots,v_{i-1}$ have already been chosen, $i<r$. (If $i=r$, the unique vertex in~$\V_{\rho}\setminus\{v_1,\ldots,v_{r-1}\}$ is certainly taken for $v_r$.)
Consider all values $|\ell_{w_1w_2}|_{\varphi}$, where $w_1,w_2\in \V_{\rho}\setminus\{v_1,\ldots,v_{i-1}\}$, and choose the maximum of them. Let the maximum be attained at a pair $(w_1^0,w_2^0)$. Then we take for $v_i$ the vertex~$w_1^0$. Now we put $\rho\cup\{v_i\}\succ\rho\cup\{v_j\}$ whenever $i<j$.

Let us check that the ordering~$\succ$ constructed by the above inductive procedure satisfies condition~(ii)${}'$ or condition~(ii)${}''$ (depending on the evenness of~$n$). Suppose that $v=v_i$ in the above notation for the vertices in~$\V_{\rho}$. Then $\V_{\rho}(v)\cup\{v\}=\V_{\rho}\setminus\{v_1,\ldots,v_{i-1}\}$. Take for~$u$ the vertex $w_2^0$. Then for any  $w_1,w_2\in \V_{\rho}(v)\cup\{v\}$, we have $|\ell_{w_1w_2}|_{\varphi}\le |\ell_{uv}|_{\varphi}$, i.\,e., $\varphi(\ell_{w_1w_2}/\ell_{uv})\ne\infty$.

The following proposition is the main fact on the ordering~$\succ$.

\begin{propos}\label{propos_main}
Let $\sigma$ and $\tau$ be simplices of~$K_{\varphi}$ such that $\dim\sigma=\dim\tau>\frac{n}{2}$ and $\mu(\sigma)=\mu(\tau)$. Then $\sigma\cup\tau\in K_{\varphi}$. 
\end{propos}
 
\begin{proof}
We may assume that $\sigma\succ\tau$.
Suppose, $\dim\sigma=\dim\tau=k$. We put $\sigma_{k}=\sigma$, and successively put $\sigma_i=\mu(\sigma_{i+1})$, $i=k-1,\ldots,1,0$. 
We denote the vertices of~$\sigma$ by~$v_0,\ldots,v_k$ so that $\sigma_i=\{v_0,\ldots,v_i\}$, $i=0,\ldots,k$. We denote by $u_k$ the vertex of~$\tau$ opposite to the facet $\mu(\tau)=\sigma_{k-1}$; then $\tau=\{v_0,\ldots,v_{k-1},u_k\}$. Since $\sigma\succ\tau$, we have $u_k\in \V_{\sigma_{k-1}}(v_k)$.
We need to prove that $\varphi(\ell_{u_kv_k})\ne\infty$. If we prove this, we will obtain that $\varphi$ is finite on all squares of edge lengths of the simplex~$\sigma\cup\tau$, hence, $\sigma\cup\tau\in K_{\varphi}$. We consider two cases.

\textit{Case 1:  $n$ is odd.\/} Let $n=2r-1$. For $i=1,\ldots,k-1$, consider condition~(ii)${}'$ for the simplex $\sigma_i=\sigma_{i-1}\cup\{v_i\}$. Obviously, the set $\V_{\sigma_{i-1}}(v_i)$ is non-empty, since it contains the vertex~$v_{i+1}$. Hence property~(ii)${}'$ implies that there is a vertex~$u_i\in \V_{\sigma_{i-1}}(v_i)$ such that 
\begin{equation}\label{eq_finite_odd}
\varphi\left(\frac{\ell_{w_1w_2}}{\ell_{u_iv_i}}\right)\ne\infty
\end{equation}
for all $w_1,w_2\in \V_{\sigma_{i-1}}(v_i)\cup\{v_i\}$.
 
Let us prove that $u_j,v_j\in  \V_{\sigma_{i-1}}(v_i)$ whenever $1\le i<j\le k$. Let $w$ be either $u_j$ or~$v_j$. Since $w\in  \V_{\sigma_{j-1}}$, we see that $\sigma_{j-1}\cup\{w\}\in K_{\varphi}$ and $\mu(\sigma_{j-1}\cup\{w\})=\sigma_{j-1}$. Condition~(i) easily implies that $\mu(\sigma_{i-1}\cup\{w\})=\sigma_{i-1}$ and $\mu(\sigma_i\cup\{w\})=\sigma_i$. Hence $w\in \V_{\sigma_{i-1}}$ and $\sigma_i\succ\sigma_{i-1}\cup\{w\}$. Therefore $w\in \V_{\sigma_{i-1}}(v_i)$. Thus~\eqref{eq_finite_odd} implies that 
\begin{equation}\label{eq_finite_odd_concrete}
\varphi\left(\frac{\ell_{u_jv_j}}{\ell_{u_iv_i}}\right)\ne\infty,\quad
\varphi\left(\frac{\ell_{u_iv_j}}{\ell_{u_iv_i}}\right)\ne\infty,\quad
\varphi\left(\frac{\ell_{u_iu_j}}{\ell_{u_iv_i}}\right)\ne\infty
\end{equation}
whenever $1\le i<j\le k$.

Besides, $\varphi$ is finite on $\ell_{v_iv_j}$ and $\ell_{v_iu_j}$ whenever $0\le i<j\le k$, since $\{v_i,v_j\}$ and $\{v_i,u_j\}$ are edges of~$K_{\varphi}$.

Assume that $\varphi(\ell_{u_kv_k})=\infty$. By~\eqref{eq_finite_odd_concrete}, we have $\varphi\left(\frac{\ell_{u_iv_i}}{\ell_{u_kv_k}}\right)\ne 0$ whenever $0<i<k$. Hence, $\varphi(\ell_{u_iv_i})=\infty$ for $i=1,\ldots,k$.

Since $k>\frac{n}{2}$, we have $k\ge r$. Consider the Cayley--Menger relation for the $n+2$ points $u_1,\ldots,u_r,v_0,v_1,\ldots,v_r$. We have
$$\left|\,\,\,\,
\begin{matrix}
0&1&1&\cdots& 1 &  \vline &1 & 1 & 1& \cdots & 1\\
1&0&\ell_{u_1u_2}&\cdots&\ell_{u_1u_r}&\vline & \ell_{u_1v_0} & \ell_{u_1v_1} & \ell_{u_1v_2} &\cdots &\ell_{u_1v_r}\\
1&\ell_{u_1u_2}&0&\cdots&\ell_{u_2u_r}& \vline &\ell_{u_2v_0} & \ell_{u_2v_1} & \ell_{u_2v_2} &\cdots &\ell_{u_2v_r}\\
\vdots &\vdots & \vdots & \ddots &\vdots& \vline &\vdots & \vdots &\vdots & \ddots &\vdots\\
1&\ell_{u_1u_r}&\ell_{u_2u_r}&\cdots&0& \vline &\ell_{u_rv_0} & \ell_{u_rv_1} & \ell_{u_rv_2} &\cdots &\ell_{u_rv_r}\\
\hline
1&\ell_{u_1v_0}&\ell_{u_2v_0}&\cdots& \ell_{u_rv_0}&\vline  &  0 &\ell_{v_0v_1} & \ell_{v_0v_2} &  \cdots & \ell_{v_0v_r}\\
1&\ell_{u_1v_1}&\ell_{u_2v_1}&\cdots& \ell_{u_rv_1} &\vline &  \ell_{v_0v_1} & 0& \ell_{v_1v_2} &  \cdots & \ell_{v_1v_r}\\
1&\ell_{u_1v_2}&\ell_{u_2v_2}&\cdots& \ell_{u_rv_2} &\vline &  \ell_{v_0v_2} & \ell_{v_1v_2} & 0 &  \cdots & \ell_{v_2v_r}\\
\vdots &\vdots & \vdots & \ddots &\vdots& \vline &\vdots & \vdots &\vdots & \ddots &\vdots\\
1&\ell_{u_1v_r}&\ell_{u_2v_r}&\cdots& \ell_{u_rv_r} & \vline & \ell_{v_0v_r} & \ell_{v_1v_r} &\ell_{v_2v_r}  &\cdots & 0
\end{matrix}
\,\,\,\,\right|=0.
$$
For $i=1,\ldots,r$, we divide both the $(r+1)$st row and the $(r+1)$st column of this determinant by $\ell_{u_iv_i}$. To simplify formulae, we put $h_i=\ell_{u_iv_i}$. We obtain
\begin{equation}\label{eq_CM_divided}
\extrarowheight=4pt
\left|\,\,\,\,
\begin{matrix}
0&\frac{1}{h_1}&\frac{1}{h_2}&\cdots& \frac{1}{h_r} &\vline & 1 & 1 & 1& \cdots & 1\\
\frac{1}{h_1}&0&\frac{\ell_{u_1u_2}}{h_1h_2}&\cdots&\frac{\ell_{u_1u_r}}{h_1h_r}& \vline &\frac{\ell_{u_1v_0}}{h_1} & 1 & \frac{\ell_{u_1v_2}}{h_1} &\cdots &\frac{\ell_{u_1v_r}}{h_1}\\
\frac{1}{h_2}&\frac{\ell_{u_1u_2}}{h_1h_2}&0&\cdots&\frac{\ell_{u_2u_r}}{h_2h_r}& \vline & \frac{\ell_{u_2v_0}}{h_2} & \frac{\ell_{u_2v_1}}{h_2} & 1 &\cdots &\frac{\ell_{u_2v_r}}{h_2}\\
\vdots &\vdots & \vdots & \ddots &\vdots& \vline &\vdots & \vdots &\vdots & \ddots &\vdots\\
\frac{1}{h_r}&\frac{\ell_{u_1u_r}}{h_1h_r}&\frac{\ell_{u_2u_r}}{h_2h_r}&\cdots&0& \vline &\frac{\ell_{u_rv_0}}{h_r} & \frac{\ell_{u_rv_1}}{h_r} & \frac{\ell_{u_rv_2}}{h_r} &\cdots &1\vphantom{\Bigl(}\\
\hline
1&\frac{\ell_{u_1v_0}}{h_1}&\frac{\ell_{u_2v_0}}{h_2}&\cdots& \frac{\ell_{u_rv_0}}{h_r} &\vline &  0 &\ell_{v_0v_1} & \ell_{v_0v_2} &  \cdots & \ell_{v_0v_r}\\
1&1&\frac{\ell_{u_2v_1}}{h_2}&\cdots& \frac{\ell_{u_rv_1}}{h_r} & \vline & \ell_{v_0v_1} & 0& \ell_{v_1v_2} &  \cdots & \ell_{v_1v_r}\\
1&\frac{\ell_{u_1v_2}}{h_1}&1&\cdots& \frac{\ell_{u_rv_2}}{h_r} & \vline & \ell_{v_0v_2} & \ell_{v_1v_2} & 0 &  \cdots & \ell_{v_2v_r}\\
\vdots &\vdots & \vdots & \ddots &\vdots& \vline &\vdots & \vdots &\vdots & \ddots &\vdots\\
1&\frac{\ell_{u_1v_r}}{h_1}&\frac{\ell_{u_2v_r}}{h_2}&\cdots& 1 &\vline &  \ell_{v_0v_r} & \ell_{v_1v_r} &\ell_{v_2v_r}  &\cdots & 0
\end{matrix}\,\,\,\,\right|=0.
\end{equation}
Now we apply the place~$\varphi$ to all entries of the determinant. 
Since $\varphi(h_i)=\infty$, we have  $\varphi\left(\frac{1}{h_i}\right)=0$. Inequalities~\eqref{eq_finite_odd_concrete} imply that $\varphi$ is finite on all entries of the determinant. Besides, we have
\begin{align*}
&\varphi\left(\frac{\ell_{u_iu_j}}{h_ih_j}\right)=
\varphi\left(\frac{\ell_{u_iu_j}}{\ell_{u_iv_i}}\right)\varphi\left(\frac{1}{\ell_{u_jv_j}}\right)=0,&&1\le i<j\le r,\\
&\varphi\left(\frac{\ell_{u_jv_i}}{h_j}\right)=\frac{\varphi(\ell_{u_jv_i})}{\varphi(\ell_{u_jv_j})}=0,&&0\le i<j\le r.
\end{align*}
Hence we obtain
\begin{eqnarray*}
\left|\,\,\,\,
\begin{matrix}
\text{{\Huge$0$}} & \vline &
\begin{matrix}
1 &  &  & \\
 & 1 &  \smash{\lefteqn{\hphantom{l}\text{\Huge$*$} }} & \\
   & & \ddots & \\
\smash{\lefteqn{\hphantom{l}\stackrel{\text{\Huge$0$}}{\vphantom{x}} }} &   & & 1
\end{matrix}
\\
\hline
\begin{matrix}
1 &  &  & \\
& 1 &   \smash{\lefteqn{\hphantom{l}\text{\Huge$0$} }}& \\
 & & \ddots &  \\
 \smash{\lefteqn{\hphantom{l}\stackrel{\text{\Huge$*$}}{\vphantom{x}} }} & & & 1
\end{matrix} & \vline
 &
 \text{\Huge$*$}
\end{matrix}\,\,\,\,\right|=0,
\end{eqnarray*}
where $*$ stand for finite elements of~$F$. This is impossible, since the left-hand side is equal to $\pm 1$. This contradiction proves that $\varphi(\ell_{u_kv_k})\ne\infty$. Hence $\sigma\cup\tau\in K_{\varphi}$.

\textit{Case 2: $n$ is even.} Let $n=2r-2$. As in the previous case, condition~(ii)${}''$ implies that there exist vertices~$u_i\in \V_{\sigma_{i-1}}(v_i)$, $i=2,\ldots,k-1$, such that inequalities~\eqref{eq_finite_odd_concrete} hold whenever $2\le i<j\le k$. (However, we cannot find a vertex~$u_1$ with a similar property.) Besides, $\varphi$ is finite on $\ell_{v_iv_j}$ and $\ell_{v_iu_j}$ whenever $i<j$.

Now we introduce a new point $z=(x_{z,1},\ldots,x_{z,n})$, $x_{z,i}\in \Q(\x_{\V})$. 
If  $\varphi(\ell_{v_0v_1})\ne 0$, we put $z=v_0$, i.\,e., $x_{z,i}=x_{v_0,i}$. 

Suppose $\varphi(\ell_{v_0v_1})=0$. Then we put
\begin{align*}
x_{z,1}&=x_{v_0,1}+a,\\
x_{z,i}&=x_{v_0,i},\quad i\ne 1,
\end{align*}
where 
$$
a=\left\{
\begin{aligned}
&\frac{1}{x_{v_0,1}-x_{v_1,1}}&&\text{if $|x_{v_0,1}-x_{v_1,1}|_{\varphi}>1$ and $\ch F\ne 2$},\\
&\hspace{.45mm}x_{v_1,1}-x_{v_0,1}&&\text{if $|x_{v_0,1}-x_{v_1,1}|_{\varphi}=1$ and $\ch F\ne 2$},\\
&\qquad \ \hspace{.65mm} 1&&\text{if $|x_{v_0,1}-x_{v_1,1}|_{\varphi}<1$ or $\ch F=2$}.
\end{aligned}
\right.
$$

\begin{lem}\label{lem_z}
We have $\varphi(\ell_{zv_1})\ne 0$, and $\varphi(\ell_{zw})\ne\infty$ for $w=v_1,\ldots,v_r,u_2,\ldots,u_r$.
\end{lem}
\begin{proof}
If $\varphi(\ell_{v_0v_1})\ne 0$, the assertion of the lemma holds, since $z=v_0$ and $\{v_0,w\}$ are edges of~$K_{\varphi}$, $w=v_1,\ldots,v_r,u_2,\dots,u_r$. 

Suppose,  $\varphi(\ell_{v_0v_1})=0$, $w\in\{v_1,\ldots,v_r,u_2,\dots,u_r\}$. We have $\{v_0\}\succ\{v_1\}$. It follows easily from condition~(i) that  $\{v_0\}\succ\{w\}$, and  if $w\ne v_1$, $\{v_0,v_1\}\succ\{v_0,w\}$. By condition~(iii)${}''$, we obtain that $ |x_{v_0,1}-x_{w,1}|_{\varphi}\le |x_{v_0,1}-x_{v_1,1}|_{\varphi}$.

We have
$$
\ell_{zw}=\ell_{v_0w}+2a(x_{v_0,1}-x_{w,1})+a^2.
$$

Suppose $|x_{v_0,1}-x_{v_1,1}|_{\varphi}>1$ and $\ch F\ne 2$. Then $\varphi(a)=0$. Hence
$$
\varphi(\ell_{zw})=\varphi(\ell_{v_0w})+2\varphi\left(\frac{x_{v_0,1}-x_{w,1}}{x_{v_0,1}-x_{v_1,1}}\right)\ne\infty.
$$ 
Besides, $\varphi(\ell_{zv_1})=2\ne 0$.

Suppose $|x_{v_0,1}-x_{v_1,1}|_{\varphi}=1$ and $\ch F\ne 2$. Then $|x_{v_0,1}-x_{w,1}|_{\varphi}\le 1$. Hence $\varphi(x_{v_0,1}-x_{v_1,1})$ is neither~$0$ nor~$\infty$, and $\varphi(x_{v_0,1}-x_{w,1})\ne \infty$. Therefore,
\begin{gather*}
\varphi(\ell_{zw})=\varphi\left(
\ell_{v_0w}+2(x_{v_1,1}-x_{v_0,1})(x_{v_0,1}-x_{w,1})+(x_{v_1,1}-x_{v_0,1})^2\right)\ne\infty,\\
\varphi(\ell_{zv_1})=-\varphi(x_{v_0,1}-x_{v_1,1})^2\ne 0.
\end{gather*}

Suppose either $|x_{v_0,1}-x_{v_1,1}|_{\varphi}<1$ or $\ch F=2$. Then $\varphi(2(x_{v_0,1}-x_{w,1}))=0$. Hence $\varphi(\ell_{zw})=\varphi(\ell_{v_0w})+1\ne\infty$,
and  $\varphi(\ell_{zv_1})=1\ne 0$.
\end{proof}

Suppose that $\varphi(\ell_{u_kv_k})=\infty$. Then it follows from~\eqref{eq_finite_odd_concrete} that $\varphi(\ell_{u_iv_i})=\infty$, $i=2,\dots,k-1$.

Since $k>\frac{n}{2}$, we have $k\ge r$. Consider the Cayley--Menger relation for the $n+2$ points $u_2,\ldots,u_r,z,v_1,\ldots,v_r$. Dividing both the $i$th row and the $i$th column of this Cayley--Menger determinant by $h_i=\ell_{u_iv_i}$, $i=2,\ldots,r$, we obtain the following analogue of equation~\eqref{eq_CM_divided}.
\begin{equation*}
\extrarowheight=4pt
\left|\,\,\,\,\begin{matrix}
0&\frac{1}{h_2}&\cdots& \frac{1}{h_r} &\vline & 1 &\vline& 1 & 1& \cdots & 1\\
\frac{1}{h_2}&0&\cdots&\frac{\ell_{u_2u_r}}{h_2h_r}& \vline & \frac{\ell_{zu_2}}{h_2} &\vline& \frac{\ell_{u_2v_1}}{h_2} & 1 &\cdots &\frac{\ell_{u_2v_r}}{h_2}\\
\vdots &\vdots & \ddots &\vdots& \vline &\vdots&\vline & \vdots &\vdots & \ddots &\vdots\\
\frac{1}{h_r}&\frac{\ell_{u_2u_r}}{h_2h_r}&\cdots&0& \vline &\frac{\ell_{zu_r}}{h_r}&\vline & \frac{\ell_{u_rv_1}}{h_r} & \frac{\ell_{u_rv_2}}{h_r} &\cdots &1\vphantom{\Bigl(}\\
\hline
1&\frac{\ell_{zu_2}}{h_2}&\cdots& \frac{\ell_{zu_r}}{h_r} &\vline &  0 &\vline&\ell_{zv_1} & \ell_{zv_2} &  \cdots & \ell_{zv_r}\vphantom{\Bigl(}\\
\hline
1&\frac{\ell_{u_2v_1}}{h_2}&\cdots& \frac{\ell_{u_rv_1}}{h_r} & \vline & \ell_{zv_1} &\vline& 0& \ell_{v_1v_2} &  \cdots & \ell_{v_1v_r}\\
1&1&\cdots& \frac{\ell_{u_rv_2}}{h_r} & \vline & \ell_{zv_2} &\vline& \ell_{v_1v_2} & 0 &  \cdots & \ell_{v_2v_r}\\
\vdots & \vdots & \ddots &\vdots& \vline &\vdots &\vline& \vdots &\vdots & \ddots &\vdots\\
1&\frac{\ell_{u_2v_r}}{h_2}&\cdots& 1 &\vline &  \ell_{zv_r} &\vline& \ell_{v_1v_r} &\ell_{v_2v_r}  &\cdots & 0
\end{matrix}\,\,\,\,
\right|=0.
\end{equation*}
Now we subtract the $(r+2)$nd row from the $(r+1)$st row, and subtract the $(r+2)$nd column from the $(r+1)$st column. We obtain  
\begin{equation}\label{eq_CM_divided_ev}
\extrarowheight=4pt
\left|\,\,\,\begin{matrix}
0&\frac{1}{h_2}&\cdots& \frac{1}{h_r} &\vline & 0 & \vline & 1 & 1& \cdots & 1\\
\frac{1}{h_2}&0&\cdots&\frac{\ell_{u_2u_r}}{h_2h_r}& \vline & p_2 &\vline & \frac{\ell_{u_2v_1}}{h_2} & 1 &\cdots &\frac{\ell_{u_2v_r}}{h_2}\\
\vdots &\vdots & \ddots &\vdots& \vline &\vdots &\vline & \vdots &\vdots & \ddots &\vdots\\
\frac{1}{h_r}&\frac{\ell_{u_2u_r}}{h_2h_r}&\cdots&0& \vline &p_r &\vline & \frac{\ell_{u_rv_1}}{h_r} & \frac{\ell_{u_rv_2}}{h_r} &\cdots &1\vphantom{\Bigl(}\\
\hline
0&p_2&\cdots& p_r &\vline &  -2\ell_{zv_1} & \vline &\ell_{zv_1} & q_2 &  \cdots & q_r\vphantom{\Bigl(}\\
\hline
1&\frac{\ell_{u_2v_1}}{h_2}&\cdots& \frac{\ell_{u_rv_1}}{h_r} & \vline & \ell_{zv_1}  & \vline & 0& \ell_{v_1v_2} &  \cdots & \ell_{v_1v_r}\\
1&1&\cdots& \frac{\ell_{u_rv_2}}{h_r} & \vline & q_2 & \vline & \ell_{v_1v_2} & 0 &  \cdots & \ell_{v_2v_r}\\
\vdots & \vdots & \ddots &\vdots& \vline &\vdots &\vline & \vdots &\vdots & \ddots &\vdots\\
1&\frac{\ell_{u_2v_r}}{h_2}&\cdots& 1 &\vline &  q_r&\vline & \ell_{v_1v_r} &\ell_{v_2v_r}  &\cdots & 0
\end{matrix}\,\,\,
\right|=0,
\end{equation}
where $p_i=\frac{\ell_{zu_i}-\ell_{u_iv_1}}{h_i}$, $q_i=\ell_{zv_i}-\ell_{v_1v_i}$. It easily follows from Lemma~\ref{lem_z} that $\varphi(p_i)=0$ and $\varphi(q_i)\ne\infty$, $i=2,\ldots,r$. As in the case of odd~$n$, we also obtain that $\varphi$ is finite on all other entries of the determinant in the left-hand side of~\eqref{eq_CM_divided_ev}, and $\varphi$ is $0$ on $\frac{1}{h_i}$, $\frac{\ell_{u_iu_j}}{h_ih_j}$, and $\frac{\ell_{u_jv_i}}{h_j}$ whenever $i<j$. 

Denote by $A$ the matrix in the left-hand side of~\eqref{eq_CM_divided_ev}. Let $B$ be the matrix $A$ with the $(r+1)$st row and the $(r+1)$st column deleted. Let $A_0$ be the matrix obtained from $A$ by replacing the central entry~$-2\ell_{zv_1}$ with~$0$. Then~\eqref{eq_CM_divided_ev} can be rewritten as
\begin{equation}\label{eq_AB}
\ell_{zv_1}\det B-\frac12\det A_0=0.
\end{equation}
Now we apply $\varphi$ to the both sides of this equation. First, $\varphi(\ell_{zv_1})\ne 0$ by Lemma~\ref{lem_z}. Second, $\varphi(B)$ is a symmetric $2r\times 2r$-matrix with finite entries such that the upper-left $r\times r$-submatrix of~$\varphi(B)$ vanishes, and the upper-right $r\times r$-submatrix of~$\varphi(B)$ is upper unitriangular. Hence $\varphi(\det(B))=\pm 1$. Finally, $A_0$ is a symmetric $(2r+1)\times (2r+1)$-matrix with zero diagonal. Hence $\frac12\det A_0$ is a polynomial with integral coefficients in the entries of~$A_0$  such that it contains exactly one of every two summands corresponding to  mutually inverse permutations in $S_{2r+1}$ without fixed points.
Hence $\varphi(\frac12\det A_0)$ is the same polynomial in the entries of~$\varphi(A_0)$. All entries of~$\varphi(A_0)$ are finite, and the upper-left $(r+1)\times(r+1)$-submatrix of~$\varphi(A_0)$ vanishes. Therefore, $\varphi(\frac12\det A_0)=0$, which contradicts equation~\eqref{eq_AB}. 

Thus $\varphi(\ell_{u_kv_k})\ne\infty$. Therefore $\sigma\cup\tau\in K_{\varphi}$.
\end{proof} 
 
\section{Proof of Main Lemma}\label{section_proof_lemma} 

Let $d$ be the dimension of~$K_{\varphi}$.
We apply to $K_{\varphi}$ the following sequence of elementary collapses. Let 
$$
\sigma^d_1\succ\sigma^d_2\succ\dots\succ\sigma^d_{s_d}
$$
be all $d$-simplices of $K_{\varphi}$. We successively collapse the pairs $(\sigma^d_1,\mu(\sigma^d_1))$,  $(\sigma^d_2,\mu(\sigma^d_2))$, $\ldots$, $(\sigma^d_{s_d},\mu(\sigma^d_{s_d}))$. Let $K_{d-1}$ be the simplicial complex obtained from~$K_{\varphi}$ by all these collapses. Let
$$
\sigma^{d-1}_1\succ\sigma^{d-1}_2\succ\dots\succ\sigma^{d-1}_{s_{d-1}}
$$
be all $(d-1)$-simplices of $K_{d-1}$. We successively collapse the pairs $(\sigma^{d-1}_1,\mu(\sigma^{d-1}_1))$,  $(\sigma^{d-1}_2,\mu(\sigma^{d-1}_2)), \ldots$, $(\sigma^{d-1}_{s_{d-1}},\mu(\sigma^{d-1}_{s_{d-1}}))$. Further we similarly successively collapse all $(d-2)$-simplices of the obtained simplicial complex~$K_{d-2}$, then all $(d-3)$-simplices of the obtained simplicial complex~$K_{d-3}$, and so on, until the dimension of the complex becomes less than or equal to~$\left[\frac{n}{2}\right]$. 
 
We need to prove that all these collapses are well defined. Assume the converse. Let $\sigma^k$ be the first simplex in the described sequence such that the pair $(\sigma^k,\mu(\sigma^k))$ cannot be collapsed, i.\,e., is not free in the subcomplex $L\subset K_{\varphi}$ obtained by all previous collapses. We have $k>\frac{n}{2}$. 

Since $(\sigma^k,\mu(\sigma^k))$ is not a free pair in~$L$ and $\dim L=k$, we see that there is a $k$-simplex $\tau^k\in L$ such that $\tau^k\ne\sigma^k$ and $\tau^k\supset\mu(\sigma^k)$. Since $\tau^k\in L$, we have $\sigma^k\succ\tau^k$. Now property~(i) of the ordering~$\succ$ easily implies that $\mu(\tau^k)=\mu(\sigma^k)$. Since $k> \frac{n}{2}$, it follows from Propostition~\ref{propos_main} that $\sigma^k\cup\tau^k\in K_{\varphi}$. 
Now property~(i) implies that $\mu(\sigma^k\cup\tau^k)=\sigma^k$. We see that there is at least one simplex $\eta^{k+1}\in K_{\varphi}$ such that $\mu(\eta^{k+1})=\sigma^k$. Among all such simplices $\eta^{k+1}$ we choose the simplex~$\eta^{k+1}_{\min}$ that is minimal with respect to~$\succ$. For each simplex $\xi^{k+2}\in K_{\varphi}$ such that $\xi^{k+2}\supset\eta^{k+1}_{\min}$, the face of~$\xi^{k+2}$ that contains  $\sigma^k$ and is distinct from~$\eta^{k+1}_{\min}$ is greater than $\eta^{k+1}_{\min}$ with respect to~$\succ$. Therefore $\mu(\xi^{k+2})\ne\eta^{k+1}_{\min}$. We obtain that no pair $(\xi^{k+2},\eta^{k+1}_{\min})$ has been collapsed. Hence the pair $(\eta^{k+1}_{\min},\sigma^k)$ has been collapsed. Therefore $\sigma^k\notin L$.  This contradiction proves that all described collapses are well defined, which completes the proof of Main Lemma.

\section{Estimates for volumes of complex polyhedra}\label{section_complex}

Let $(Z,P)$ be a simplicial polyhedron in~$\R^n$, where $Z$ is an $(n-1)$-cycle on the vertex set~$\V$. Recall that the mapping $P\colon\supp(Z)\to\R^n$ is uniquely determined by the images of the vertices of the polyhedron. Polyhedra of combinatorial type~$Z$ are in one-to-one correspondence with specialization homomorphisms $\Z[\x_{\V}]\to\R$.

Consider the complexification of the definition of a simplicial polyhedron. This means that we replace the mapping $P\colon|\supp(Z)|\to\R^n$ with a mapping $P\colon |\supp(Z)|\to\C^n$.
Such polyhedra~$(Z,P)$ will be called \textit{complex simplicial polyhedra\/} of combinatorial type~$Z$. The squares of the edge lengths $\ell_{uv}(P)$ and the oriented volume~$V_Z(P)$ are given by the same polynomials in coordinates of vertices as for real polyhedra of the same combinatorial type.  Notice that this means that $\ell_{uv}(P)$ are the squares of the \textit{orthogonal\/} lengths of edges of~$P$ given by
$$
\ell_{uv}(P)=\sum_{i=1}^n(x_{u,i}(P)-x_{v,i}(P))^2.
$$ 
 They do not coincide with the squares of the \textit{Hermitian} edge lengths given by
$$
h_{uv}(P)=\sum_{i=1}^n|x_{u,i}(P)-x_{v,i}(P)|^2.
$$ 
Here $x_{u,i}(P)$ and $x_{v,i}(P)$ are the coordinates of the points~$P(u)$ and $P(v)$ respectively. Obviously, we have $|\ell_{uv}(P)|\le h_{uv}(P)$. However, $h_{uv}(P)$ may be arbitrarily large for bounded~$\ell_{uv}(P)$.

Let $Z=\sum_{i=1}^qc_i\Delta_i$, where $c_i\in\Z$ and $\Delta_i$ are distinct oriented $(n-1)$-simplices. We put $c_{\max}=\max_{i=1,\dots,q}|c_i|$ and $c_{\Sigma}=\sum_{i=1}^q|c_i|$. Let $m$ be the number of vertices of~$\supp(Z)$.

First, let $(Z,P)$ be a real simplicial polyhedron. Let $d$ be the diameter of~$P(|\supp(Z)|$.  The oriented volume~$V_Z(P)$ can be easily written as the sum of~$c_{\Sigma}$ oriented volumes of simplices each of diameter not greater than~$d$. The absolute value of each such volume of a simplex does not exceed $\frac{d^n}{n!}$. The diameter~$d$ does not exceed the maximal edge length multiplied by~$m$. Hence we have
\begin{equation}\label{eq_est_real}
|V_Z(P)|\le \frac{c_{\Sigma} m^n}{n!}\left(\max_{\{u,v\}\in\supp(Z)}\ell_{uv}(P)\right)^{n/2}.
\end{equation}
In the same way, for a complex simplicial polyhedron~$(Z,P)$, we obtain the estimate
\begin{equation}\label{eq_est_Herm}
|V_Z(P)|\le \frac{c_{\Sigma} m^n}{n!}\left(\max_{\{u,v\}\in\supp(Z)}h_{uv}(P)\right)^{n/2}.
\end{equation} 
In fact, the constants in estimates~\eqref{eq_est_real} and~\eqref{eq_est_Herm} can be easily improved. We do not want to care about them.

A natural question is as follows. Can we obtain an analogue of estimate~\eqref{eq_est_Herm} (with some multiplicative constant~$C_Z$ depending on~$Z$) with the squares of the Hermitian edge lengths~$h_{uv}(P)$ replaced by the squares of the orthogonal edge lengths~$\ell_{uv}(P)$? This question is non-trivial, since we cannot estimate the squares of the  orthogonal lengths of diagonals from  the squares of the orthogonal lengths of edges. Another explanation why this question is non-trivial is that the quotient of the set of all complex polyhedra of the fixed combinatorial type and bounded orthogonal edge lengths by the action of $O(n,\C)$ may be non-compact. Moreover, in dimension~$2$ the answer is negative. Indeed, consider the quadrangle in~$\C^2$ with consecutive vertices $(0,0)$, $(1,-i)$, $(2,0)$, $(1,i)$. The squares of the  orthogonal lengths of all its edges vanish, while the oriented area is equal to~$2i$.

However, we can derive from Theorem~\ref{theorem_main2} the existence of the required estimate in dimensions $n\ge 3$. 

\begin{cor}\label{cor_est}
For any complex simplicial polyhedron~$(Z,P)$ in~$\C^n$, $n\ge 3$, we have 
\begin{equation*}
|V_Z(P)|\le C_Z\left(\max_{\{u,v\}\in\supp(Z)}|\ell_{uv}(P)|\right)^{n/2},
\end{equation*}
where the constant~$C_Z$ depends on the cycle~$Z$ only.
\end{cor}

\begin{proof}
Since $V_Z(P)$ and $\ell_{uv}(P)$ are homogeneous functions in the coordinates of vertices of~$P$ of degrees~$n$ and $2$ respectively, it is sufficient to prove that $|V_Z(P)|$ is bounded by some positive constant~$C_Z$ if $|\ell_{uv}(P)|\le 1$ for all $\{u,v\}\in\supp(Z)$. 
It follows from Theorem~\ref{theorem_main2} that $V_Z(P)$ is a root of an equation of the form
$$
V^{2N}+a_1(\ell)V^{2N-2}+\cdots+a_N(\ell)=0,
$$
where $a_i$ are polynomials in the squares of the orthogonal edge lengths of the polyhedron. The corollary follows, since the roots of a monic polynomial with fixed degree and bounded coefficients are bounded.
 \end{proof}

Notice, however, that we cannot give any reasonable estimate for the constant~$C_Z$. In both~\eqref{eq_est_real} and~\eqref{eq_est_Herm} the multiplicative constant does not exceed some polynomial in~$m$ multiplied by~$c_{\max}$. The following question seems to be interesting.
\begin{quest}
Does there exist an estimate of the form $C_Z\le \lambda_nc_{\max}m^{s_n}$, where the constants~$\lambda_n$ and~$s_n$ depend on the dimension~$n$ only? If yes, what is the minimal possible value of~$s_n$? 
\end{quest}

\begin{remark}
In fact, Theorem~\ref{theorem_main2} can be deduced from Corollary~\ref{cor_est}. However, we do not know any proof of Corollary~\ref{cor_est} that does not use Theorem~\ref{theorem_main2}.
\end{remark}

\section{Polyhedra with triangular $2$-faces}\label{section_triang} 

In this section we prove Theorem~\ref{theorem_main_ns}.   

Let $\F_k$, $k=1,\ldots,n$, be the subset of~$\F$ consisting of all $k$-dimensional faces. (Notice that the dimension of a face $F\in\F$ is the same for all polyhedra of combinatorial type~$\F$, since it coincides with the maximal length~$k$ of a chain $F_0< \dots< F_k=F$ in~$\F$.) The set~$\V=\F_0$ is the set of vertices. For each $F\in\F$, we denote by~$\V_F$ the set of all vertices $v\in\V$ such that $v<F$. (This does not agree with the notation used in section~\ref{section_ord}.)

In our notation, the symbol~$P$ has two meanings.  
First, $P$ is a polyhedron of combinatorial type~$\F$. Second, $P$ denotes the maximal element of the poset~$\F$. When we consider different polyhedra of the same combinatorial type, this may lead to a confusion. So we change the notation and denote the maximal element of~$\F$ by~$Q$.

If the polyhedron~$P$ is not simplicial, the coordinates $x_{v,i}$ of its vertices are no longer independent. Let~$A_{\F}$ be the commutative ring with unity given by generators $x_{v,i}$, $v\in\V$, $i=1,\ldots,n$, and relations
\begin{equation}\label{eq_det}
\begin{vmatrix}
x_{v_1,i_1}-x_{v_0,i_1}&  \ldots & x_{v_{k+1},i_1}-x_{v_0,i_1}\\
\vdots&\ddots&\vdots\\
x_{v_1,i_{k+1}}-x_{v_0,i_{k+1}}&  \ldots & x_{v_{k+1},i_{k+1}}-x_{v_0,i_{k+1}}
\end{vmatrix}=0
\end{equation}
such that $v_0,\ldots,v_k\in \V_F$ for some $F\in\F_k$, and $1\le i_1<\dots<i_{k+1}\le n$. 
In other words, $A_{\F}$ is the quotient of the ring~$\Z[\x_{\V}]$ by the ideal~$\mathfrak{a}$ generated by the left-hand sides of relations~\eqref{eq_det}.
The \textit{universal squares of the  distances\/} $\ell_{uv}$, $u,v\in\V$, are defined by usual formulae $\ell_{uv}=\sum_{i=1}^n(x_{u,i}-x_{v,i})^2$.
Let $R_{\F}\subset A_{\F}$ be  the subring generated by all $\ell_{uv}$ such that $u$ and $v$ are joined by an edge, i.\,e., there is an $F\in\F_1$ such that $u<F$ and $v<F$.

We denote by~$\Delta^{\V}$ the abstract simplex with the vertex set~$\V$, and we denote by~$\Delta^{\V_F}$ the face of~$\Delta^{\V}$ spanned by~$\V_F$. By definition, the \textit{universal oriented volume\/} of an oriented $n$-face $\Delta=\{v_0,\ldots,v_n\}$ of~$\Delta^{\V}$ is
$$
V_{\Delta}=\frac{1}{n!}
\begin{vmatrix}
x_{v_1,1}-x_{v_0,1}&  \ldots & x_{v_{n},1}-x_{v_0,1}\\
\vdots&\ddots&\vdots\\
x_{v_1,n}-x_{v_0,n}&  \ldots & x_{v_{n},n}-x_{v_0,n}
\end{vmatrix}.
$$ 
For each $n$-chain $Y=\sum_{i=1}^qc_i\Delta_i$, we put $V_Y=\sum_{i=1}^qc_iV_{\Delta_i}$ and $W_Y=2^{[\frac{n}{2}]}n!V_Y$. We have $V_Y\in A_{\F}\otimes\Q$ and $W_Y\in A_{\F}$. Obviously, $W_{\partial X}=0$ if $X\in C_{n+1}(\Delta^{\V})$. If all vertices $v_0,\ldots,v_n$ of~$\Delta$ are contained in~$\V_F$ for  some~$F\in\F_{n-1}$, then $W_{\Delta}=0$ in~$A_{\F}$. Such simplices~$\Delta$ will be called \textit{degenerate\/}. The introduced notation differs from the notation used in the previous sections. Indeed, instead of the oriented volume~$V_Z$ \textit{bounded\/} by an $(n-1)$-cycle~$Z$, we now consider the oriented volume~$V_Y$ \textit{of\/} an $n$-chain~$Y$. 

Every polyhedron~$P$ of combinatorial type~$\F$ yields the specialization homomorphism $A_{\F}\to\R$ that takes $x_{v,i}$ to the coordinates of the vertices of~$P$, $\ell_{uv}$ to the squares of the edge and diagonal  lengths  of~$P$, and $V_{\Delta}$ to the oriented volumes of the corresponding geometric simplices in~$\R^n$.

Now, we want to construct an element of $A_{\F}$ that will serve as the universal oriented volume (multiplied by~$2^{[\frac{n}{2}]}n!$) of a polyhedron of combinatorial type~$\F$. The difficulty consists in the fact that the volume of~$P$ is not a polynomial in the coordinates of vertices of~$P$. Indeed, even the volume of a convex polyhedron is not a polynomial in the coordinates of vertices. To obtain a polynomial, we need to consider the oriented volume. For non-convex polyhedra the situation with orientations is more delicate.    
For instance, the two polyhedra in Figure~\ref{fig_2pol} are combinatorially equivalent, but they cannot be taken to each other by an either orientation-preserving or orientation-reversing homeomorphism that maps the faces of the first polyhedron onto the corresponding faces of the second polyhedron.

\begin{figure}
\begin{center}
\includegraphics[scale=.04]{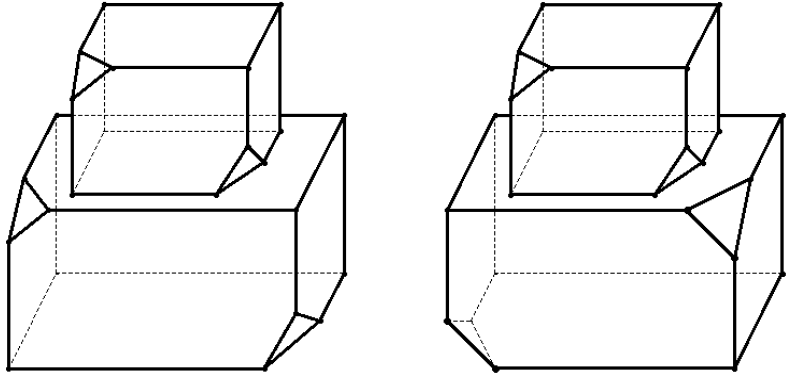}
\end{center}
\caption{Two combinatorially equivalent polyhedra}\label{fig_2pol}
\end{figure}

Consider an $n$-polyhedron~$P\subset\R^n$ of combinatorial type~$\F$. Choose arbitrarily orientations~$\mathcal{O}_F$ of all positive-dimensional faces of~$P$. (Each vertex of~$P$ is supposed to be positively oriented.) The set of orientations~$\mathcal{O}_F$ is called an \textit{omni-orientation\/} of~$P$.
Now for each two faces  $F$ and $G$ such that $\dim F=\dim G+1$ and $G\subset\partial F$, we have the incidence coefficient $\varepsilon_{F,G}$ which is equal to~$1$ if  the orientation~$\mathcal{O}_G$ coincides with the orientation of~$G$  induced by the orientation~$\mathcal{O}_F$ of~$F$, and is equal to~$-1$ if these two orientations of~$G$ are opposite to each other. For convenience, we put $\varepsilon_{F,G}=0$ if $\dim F=\dim G+1$ and  $G\not\subset\partial  F$.
For any faces $F$ and $H$ such that $\dim F=\dim H+2$, we have
\begin{equation}\label{eq_epsilon}
\sum_{G\in\F_{k-1}}\varepsilon_{F,G}\varepsilon_{G,H}=0.
\end{equation}
We shall say that the omni-oriented polyhedron~$P$ has \textit{omni-oriented combinatorial type\/}~$(\F,\{\varepsilon_{F,G}\})$.
 
The oriented volume~$V(P)$ depends only on the orientation of the maximal face of~$P$. Reversing the orientation of any other face~$G\in\F_k$, we simultaneously multiply by~$-1$ all incidence coefficients~$\varepsilon_{F,G}$, $F\in\F_{k+1}$, and ~$\varepsilon_{G,H}$, $H\in\F_{k-1}$. Two sets of numbers~$\{\varepsilon_{F,G}\}$ and~$\{\varepsilon'_{F,G}\}$ are said to be \textit{equivalent\/} if they can be taken to each other by such operations. The equivalence class~$\E$ of the set~$\{\varepsilon_{F,G}\}$ depends only on the orientation of~$P$ and is independent on the orientations of all other faces of~$P$. The pair~$(\F,\E)$ will be called  \textit{oriented combinatorial type\/} of the oriented polyhedron~$P$.

For any pair $(\F,\E)$ that can be obtained as an oriented combinatorial type of an oriented polyhedron, we shall construct an element $W_{\F,\E}\in A_{\F}$, which can be interpreted as the \textit{universal oriented volume\/} (multiplied by $2^{[\frac{n}{2}]}n!$) of an oriented polyhedron of oriented combinatorial type~$(\F,\E)$. Then we shall check that for any oriented polyhedron~$P$ of  combinatorial type~$(\F,\E)$ the corresponding specialization homomorphism $A_{\F}\to\R$ takes $W_{\F,\E}$ to~$W(P)=2^{[\frac{n}{2}]}n!V(P)$.

To compute the volume of a polyhedron it is useful to consider its triangulation. A convex polyhedron can be triangulated without adding new vertices. Moreover, we can successively choose triangulations~$T_F$ of faces~$F$ so that $T_F$ restricts to $T_G$ whenever $G\subset\partial F$. For non-convex polyhedra it is convenient to introduce the following concept of  \textit{generalized triangulation\/}. Let $\{\varepsilon_{F,G}\}$ be an arbitrary set of numbers in the equivalence class~$\E$. A generalized triangulation of a (universal) omni-oriented polyhedron of omni-oriented combinatorial type~$(\F,\{\varepsilon_{F,G}\})$ is a set of chains $Y_F\in C_{\dim F}(\Delta^{\V_F})$, $F\in\F$, such that $Y_v=\{v\}$ for every vertex $v\in\V=\F_0$, and 
\begin{equation}\label{eq_YF}
\partial Y_F=\sum_{G\in\F_{\dim F-1}}\varepsilon_{F,G}Y_G
\end{equation} 
whenever $\dim F>0$. It is easy to show that such generalized triangulations exist. Indeed,
equations~\eqref{eq_epsilon} provide that  the right-hand side of~\eqref{eq_YF} is a cycle. Hence we can successively choose the chains~$Y_F$ by induction on the dimension of~$F$.

We put $W_{\F,\E}=W_{Y_Q}$. The next lemma shows that $W_{\F,\E}$ is independent of the choice of the generalized triangulation~$\{Y_F\}$. It is easy to check that~$W_{\F,\E}$ is also independent of the choice of the set~$\{\varepsilon_{F,G}\}$ in the equivalence class~$\E$.

\begin{lem}
Let $\{Y_F\}$ and $\{Y'_F\}$ be two generalized triangulations of~$(\F,\{\varepsilon_{F,G}\})$. Then $W_{Y_Q}=W_{Y_Q'}$.
\end{lem}

\begin{proof}
For each vertex~$v\in\F_0=\V$, we have $Y_v=Y_v'=\{v\}$. Hence for each edge $F\in\F_1$, the $1$-chain $Y_F-Y_F'$ is a cycle with support in~$\Delta^{\V_F}$. Therefore there exists a chain $S_F\in C_2(\Delta^{\V_F})$ such that  $\partial S_F=Y_F-Y_F'$. Further, consecutively for $k=2,\ldots,n$, we can construct chains  $S_F\in C_{k+1}(\Delta^{\V_F})$, $F\in\F_k$, such that  
$$
\partial S_F=Y_F-Y'_F -\sum_{G\in\F_{k-1}}\varepsilon_{F,G}S_G,
$$
since the right-hand sides of these formulae are cycles with supports in the corresponding simplices~$\Delta^{\V_F}$. Then
$$
Y_Q-Y_Q'=\partial S_Q+\sum_{F\in\F_{n-1}}\varepsilon_{Q,F}S_F.
$$
Every chain $S_F$, $F\in\F_{n-1}$, is a linear combination of $n$-simplices $\{v_0,\ldots,v_n\}$ such that $v_0,\ldots,v_n$ are contained in~$F$. Therefore $W_{S_F}=0$ in~$A_{\F}$. Besides, $W_{\partial S_Q}=0$. Thus $W_{Y_Q}=W_{Y'_Q}$.
\end{proof}

\begin{lem}
Let $P$ be an omni-oriented polyhedron of omni-oriented combinatorial type~$(\F,\{\varepsilon_{F,G}\})$. Then the corresponding specialization homomorphism $A_{\F}\to\R$ takes $W_{Y_Q}$ to~$W(P)$.
\end{lem}

\begin{proof}
Consider a triangulation~$K$ of the polyhedron~$P$ such that all faces of~$P$ are subcomplexes of~$K$. (We do not require that the vertex set~$\mathcal{W}$ of~$K$ coincides with $\V$, hence, such triangulation always exists.) We have $\V\subset\mathcal{W}$, hence, $\Delta^{\V}\subset\Delta^{\mathcal W}$. 

To simplify the notation, we shall denote the images of~$W_{Y_F}$ under the specialization homomorphism $A_{\F}\to\R$ by~$W_{Y_F}$ again. The natural embedding $\mathcal W\subset\R^n$ provides that the oriented volume~$V_X\in\R$ and the corresponding number~$W_X$ is defined for every chain $X\in C_n(\Delta^{\mathcal W})$.

Let $F\in\F_k$.  We denote by~$\mathcal W_F$ the set of all vertices of~$K$ belonging to~$F$. We denote by~$X_F$ the sum of all $k$-simplices of~$K$ contained in~$F$ with orientations induced by~$\mathcal{O}_F$. Then $X_F\in C_k(\Delta^{\mathcal W_F})$. It is easy to see that  $\partial X_F=\sum_{G\in\F_{k-1}}\varepsilon_{F,G}X_G$ if $k>0$. Since, $X_Q$ is the sum of all positively oriented $n$-simplices of~$K$, we see that $W(P)=W_{X_Q}$. 

Now, as in the proof of the previous lemma, we consecutively construct chains  $S_F\in C_{\dim F+1}(\Delta^{\mathcal W_F})$ such that  
$$
\partial S_F=Y_F-X_F -\sum_{G\in\F_{\dim F-1}}\varepsilon_{F,G}S_G,
$$
and obtain that $W_{Y_Q}=W_{X_Q}=W(P)$.
\end{proof}

Now we are ready to formulate a universal algebraic version of Theorem~\ref{theorem_main_ns}.

\begin{theorem}\label{theorem_alg_ns}
Let $(\F,\E)$ be a pair that can be obtained as the oriented combinatorial type of an oriented $n$-polyhedron. Then the element $W_{\F,\E}$ is integral over the ring~$R_{\F}$.
\end{theorem}

This theorem yields a monic polynomial relation of the form~\eqref{eq_rel_ns} that holds for all oriented polyhedra of the given oriented combinatorial type~$(\F,\E)$. The required polynomial relation that holds for all polyhedra of combinatorial type~$\F$ is obtained by multiplying all such relations for different~$\E$. So Theorem~\ref{theorem_main_ns} follows from Theorem~\ref{theorem_alg_ns}.

To apply Lemma~\ref{lem_place} we would need to embed the ring~$A_{\F}$ into a field. Unfortunately, the author does not know whether the ring~$A_{\F}$ can have zero divisors. Hence we take the quotients of~$A_{\F}$ by prime ideals, and use the following algebraic lemma. 

\begin{lem}\label{lem_prime}
Let $A$ be a Noetherian commutative ring, let $R\subset A$ be a subring, and let~$W\in A$. Suppose that the image of~$W$ in~$A/\mathfrak{p}$ is integral over~$R/(R\cap \mathfrak{p})$ for every prime ideal $\mathfrak{p}\lhd A$. Then $W$ is integral over~$R$. 
\end{lem}

For the convenience of the reader, we give the proof of this lemma. We thank S.~O.~Gorchinsky for communicating this proof to us. 

\begin{proof}
Since $A$ is Noetherian, it follows that $A$ contains finitely many minimal (with respect to inclusion) prime ideals~$\mathfrak{p}_1,\ldots,\mathfrak{p}_s$, and any prime ideal $\mathfrak{p}\lhd A$ contains one of the ideals~$\mathfrak{p}_i$ (cf.~\cite[p.~47]{Eis95}). 
For every $i=1,\ldots,s$, the image of~$W$ in~$A/\mathfrak{p}_i$ is integral over~$R/(R\cap \mathfrak{p}_i)$. It follows that there exists a monic polynomial $f_i\in R[t]$ such that $f_i(W)\in\mathfrak{p}_i$. Let $f$ be the product of the polynomials~$f_i$. Then $f$ is monic and $f(W)$ belongs to the intersection $\bigcap_{i=1}^s\mathfrak{p}_i$, which coincides with the set of all nilpotent elements of~$A$ (cf.~\cite[p.~71]{Eis95}). Hence there exists a positive integer~$q$ such that $f^q(W)=0$. The lemma follows, since $f^q$ is monic.
\end{proof}

Let $\mathfrak{p}\lhd A_{\F}$ be a prime ideal. For the sake of simplicity we denote the images of the elements~$x_{v,i}$, $\ell_{uv}$, and~$W_{\F,\E}$ in~$A/\mathfrak{p}$ again by~$x_{v,i}$, $\ell_{uv}$, and~$W_{\F,\E}$ respectively.
Let $E_{\F,\mathfrak{p}}$ be the quotient field of the ring~$A/\mathfrak{p}$. 
Consider a place $\varphi\colon E_{\F,\mathfrak{p}}\to L\cup\{\infty\}$.
As in section~\ref{section_scheme}, let $K_{\varphi}\subset\Delta^{\V_F}$ be the full subcomplex spanned by all edges $\{u,v\}$ such that $\varphi(\ell_{uv})\ne\infty$.
The proof of Theorem~\ref{theorem_alg_ns} is based on the following generalization of Main Lemma.

\begin{lem}\label{lem_gm}
Let $\varphi\colon E_{\F,\mathfrak{p}}\to L\cup\{\infty\}$ be a place and let~$F\in\F_k$. Then the simplicial complex~$K_{\varphi}\cap\Delta^{\V_F}$ collapses on a subcomplex of dimension less than or equal to~$\bigl[\frac{k}{2}\bigr]$.
\end{lem} 
\begin{cor}\label{cor_main_ns}
If $k\ge 3$, then $H_{k-1}(K_{\varphi}\cap\Delta^{\V_F})=0$.
\end{cor}
\begin{proof}
For simplicity, we put $E=E_{\F,\mathfrak{p}}$.
Consider all points $v\in \V$ as the points $v=(x_{v,1},\ldots,x_{v,n})\in E^n$. Since determinants~\eqref{eq_det} vanish for the points $v_i$ in~$\V_F$, we see that all points $v\in\V_F$ lie in a $k$-dimensional affine subspace $\Pi\subset E^n$. Let $(e_1,\dots,e_n)$ be the standard basis of~$E^n$.
The squares of the distances~$\ell_{uv}$ are determined with respect to the natural inner product in~$E^n$ such that $(e_i,e_j)=\delta_{ij}$. 
Restrict this inner product to the vector subspace $U\subset E^n$ parallel to~$\Pi$. 
For an appropriate extension $E'\supset E$, we can find a basis $f_1,\ldots,f_k$ of $U_{E'}=U\otimes E'$ such that $(f_i,f_j)=\delta_{ij}$. Choose an arbitrary point $O\in\Pi$ and consider the affine coordinate system in~$\Pi_{E'}$ with origin~$O$ and basis $f_1,\ldots,f_k$.
 We denote the coordinates of $v\in\V_{F}$ in this coordinate system by $(y_{v,1},\ldots,y_{v,n})$; then $y_{v,i}\in E'$.
 Since $(f_i,f_j)=\delta_{ij}$, we have
\begin{equation}\label{eq_xy}
\ell_{uv}=\sum_{j=1}^n(x_{u,j}-x_{v,j})^2=\sum_{i=1}^k(y_{u,i}-y_{v,i})^2
\end{equation}
for every $u,v\in\V_{F}$.

Choose an arbitrary extension $\varphi'\colon E'\to L'\cup\{\infty\}$ of the place~$\varphi$.

Now consider independent variables $z_{v,i}$, $v\in\V_F$, $i=1,\ldots,k$. We denote the set of these variables by~$\mathbf{z}_F$ and the field of rational functions in these variables by~$\Q(\mathbf{z}_F)$. The correspondence 
$
z_{v,i}\mapsto y_{v,i}
$ uniquely extends to the place $\psi\colon \Q(\mathbf{z}_F)\to E'\cup\{\infty\}$. Consider the composite place
$$
\Phi\colon \Q(\mathbf{z}_F)\xrightarrow{\psi} E'\cup\{\infty\}\xrightarrow{\varphi'} L'\cup\{\infty\},
$$ 
and the full subcomplex $K_{\Phi}\subset\Delta^{\V_F}$ spanned by all edges $\{u,v\}$ such that $\Phi$ is finite on
the element $l_{uv}=\sum_{i=1}^k(z_{u,i}-z_{v,i})^2$. Applying Main Lemma to the place~$\Phi$, we obtain that $K_{\Phi}$ collapses on a subcomplex of dimension less than or equal to~$[\frac{k}{2}]$.
By~\eqref{eq_xy}, we have $\psi(l_{uv})=\ell_{uv}$. Hence $K_{\varphi}\cap\Delta^{\V_F}=K_{\Phi}$, which completes the proof of the lemma.
\end{proof}

\begin{proof}[Proof of Theorem~\ref{theorem_main_ns}]
The ring~$A_{\F}$ is Noetherian. By Lemma~\ref{lem_prime}, we need to prove that $W_{\F,\E}$ is integral over $R_{\F}/(R_{\F}\cap\mathfrak{p})$ for every prime ideal $\mathfrak{p}\lhd A_{\F}$. Hence we need to prove that $\varphi(W_{\F,\E})\ne\infty$ for every place $\varphi\colon E_{\F,\mathfrak{p}}\to L\cup\{\infty\}$ that is finite on~$R_{\F}/(R_{\F}\cap\mathfrak{p})$. 

Let $\{\varepsilon_{F,G}\}$ be a set of numbers in the equivalence class~$\E$. We shall construct a generalized triangulation~$\{Y_F\}$ of an omni-oriented polyhedron of omni-oriented oriented combinatorial type~$(\F,\{\varepsilon_{F,G}\})$ such that $\supp(Y_F)\subset K_{\varphi}\cap\Delta^{\V_F}$. For every vertex $v\in\F_0=\V$, we put $Y_v=\{v\}$. For every edge $e\in\F_1$ oriented from $u$ to~$v$, we put $Y_e=\{u,v\}$. Every $2$-face $F\in\F_2$ is triangular, i.\,e., contains exactly $3$ vertices. Let $u,v,w$ be the vertices of~$F$ listed so that to give the orientation~$\mathcal{O}_F$ of~$F$. Then we put $Y_F=\{u,v,w\}$. Obviously, conditions~\eqref{eq_YF} hold so far. If vertices~$u$ and~$v$ are joined by an edge $e\in\F_1$, then $\ell_{uv}\in R_{\F}/(R_{\F}\cap\mathfrak{p})$ and hence $\varphi(\ell_{uv})\ne\infty$. Therefore the supports of all chains $Y_F$, $\dim F\le 2$, are contained in~$K_{\varphi}$.

Further, we proceed by induction on~$k$. Let $k\ge 3$. Suppose that the chains $Y_G$, $G\in\F_{k-1}$,  have already been constructed, and consider $F\in\F_k$. The support of the $(k-1)$-chain $Z_F=\sum_{G\in\F_{k-1}}\varepsilon_{F,G}Y_G$ is contained in  $K_{\varphi}\cap\Delta^{\V_F}$. 
Hence it follows from Corollary~\ref{cor_main_ns} that there is a $k$-chain~$Y_F$ with support contained in $K_{\varphi}\cap\Delta^{\V_F}$ such that $\partial Y_F=Z_F$. 

Since $\{Y_F\}$ is a generalized triangulation, we have $W_{\F,\E}=W_{Y_Q}$. 
The Cayley--Menger formula implies that $\varphi(W_{\Delta})\ne\infty$ for all $n$-simplices $\Delta\in K_{\varphi}$. Since $\supp(Y_Q)\subset K_{\varphi}$, we obtain that $\varphi(W_{\F,\E})=\varphi(W_{Y_Q})\ne \infty$. 
\end{proof}

If in the previous argument we additionally assume that the place~$\varphi$ is finite on the squares of the diagonal lengths of all $2$-faces of~$P$, then we can waive the requirement that all $2$-faces of~$P$ are triangular. So we obtain that for any $n$-polyhedron~$P$, there exists a relation of the form
$$
W^{2N}+b_1(\ell,d)W^{2N-2}+\dots+b_N(\ell,d)=0,
$$
where $b_i(\ell,d)$ are polynomials with integral coefficients in the squares of the edge lengths of~$P$ and the squares of the diagonal lengths of $2$-faces of~$P$. This implies Corollary~\ref{cor_flex_ns}.


\begin{thebibliography}{99}



\bibitem{Bri97} R. Bricard, \textit{M\'emoire sur la th\'eorie de l'octaedre articule\/}, J. Math. Pures Appl. \textbf{5}:3 (1897), 113--148.

\bibitem{Cay41} A. Cayley, \textit{On a theorem in the geometry of position\/}, Cambridge Math. J. \textbf{2} (1841), 267--271.

\bibitem{Con77} R. Connelly, \textit{A counterexample to the rigidity conjecture for polyhedra\/}, Inst. Hautes \'Etudes Sci. Publ. Math. \textbf{47} (1977), 333--338.

\bibitem{Con78} R. Connelly, \textit{Conjectures and open questions in rigidity\/}, Proc. Internat. Congress Math. (Helsinki, 1978), Acad. Sci. Fennica, Helsinki, 1980, 407--414.

\bibitem{CSW97} R. Connelly, I. Sabitov, A. Walz, \textit{The Bellows Conjecture\/}, Beitr. Algebra Geom. \textbf{38}:1, 1--10.

\bibitem{Ede95} H. Edelsbrunner, \textit{Algebraic decomposition of non-convex polyhedra\/},  36th Annual Symposium on Foundations of Computer Science (FOCS'95), 248--257, 1995.

\bibitem{Eis95} D. Eisenbud, \textit{Commutative Algebra with a View Toward Algebraic Geometry\/}, Graduate Texts Math. \textbf{150}, Springer, 1995.




\bibitem{Gai11} A. A. Gaifullin, \textit{Sabitov polynomials for volumes of polyhedra in dimension~$4$\/}, arXiv: 1108.6014, 2011.



\bibitem{Lan72} S. Lang, \textit{Introduction to Algebraic Geometry\/}, Addison-Wesley, Massachusetts, 1972.

\bibitem{Men28} K. Menger, \textit{Untersuchungen \"uber allgemeine Metrik\/}, Math. Ann. \textbf{100}:1 (1928), 75--163.

\bibitem{Men31} K. Menger, \textit{New Foundation of Euclidean Geometry\/}, Amer. J. Math. \textbf{53}:4 (1931), 721--745.

\bibitem{Pak08} I. Pak, \textit{Lectures on Discrete and Polyhedral Geometry\/}, monograph draft, available at \texttt{http:/\!/ www.math.ucla.edu/$\sim$pak/book.htm}, 2010.

\bibitem{Sab96} I.\,Kh. Sabitov, \textit{Volume of a polyhedron as a function of its metric\/}, Fundamental and Applied Math. \textbf{2}:4 (1996), 1235--1246 (in Russian). 



\bibitem{Sab98} I.\,Kh. Sabitov, \textit{The volume  as a  metric invariant of polyhedra\/}, Discrete Comput. Geom. \textbf{20}:4 (1998), 405--425. 



\bibitem{Sta00} H. Stachel, \textit{Flexible cross-polytopes in the Euclidean $4$-space\/}, J. Geom. Graph. \textbf{4}:2 (2000), 159--167.

\end{thebibliography}
\end{document}